\documentclass[11pt]{amsart}
\usepackage{amssymb,xcolor}
\usepackage[all]{xy}
\usepackage{fullpage}
\newcommand{\PP}{\mathbb P}
\newcommand{\lra}{\longrightarrow}
\newcommand\CC{{\mathbb C}}
\newcommand\ZZ{{\mathbb Z}}
\newcommand\QQ{{\mathbb Q}}
\newcommand\mo{{\mathcal O}}
\newcommand\mn{{\mathcal N}}
\newcommand\tx{{\tilde X_{16}}}
\newcommand\ttx{{\tilde X_{5}}}
\newcommand\ty{{\tilde Y}}
\newcommand\ts{{S}}
\newcommand\tts{S_1}
\newcommand\plra{\dashrightarrow}
\newcommand{\WWW}{W} 
\newcommand\moty{{\mathcal O_{\tilde Y}}}

\newcommand\mb{{\mathcal B}}              
\newcommand\mq{{\mathfrak q}}
\newcommand\mc{{\mathcal C}}             
\newcommand\xs{{\mathcal V}_{10}}

\newcommand\exactlytensingularities{A1}

\newcommand\zanurzeniebordygi{A2}
\newcommand\exactlyfortysixsingularities{A3}

\newcommand\isolatedsingularities{A4}

\DeclareMathOperator{\sing}{sing}
\DeclareMathOperator{\rank}{rank}

\begin{document}

\title{On Calabi--Yau threefolds associated to a web of quadrics}
 \author{S{\l}awomir Cynk}
 \address{Institute of Mathematics, Jagiellonian University,
{\L}ojasiewicza 6, 30-348 Krak\'ow, Poland 
}
\address{Institute of Mathematics of the
Polish Academy of Sciences, ul. \'Sniadeckich 8, 00-956 Warszawa,
P.O. Box 21, Poland}
 \email{slawomir.cynk@uj.edu.pl}
 \author {S{\l}awomir Rams}
\address
{Institute of Mathematics,
Jagiellonian University,
ul. {\L}ojasiewicza 6,
30-348 Krak\'ow,
Poland}

\email{slawomir.rams@uj.edu.pl}
\thanks{Research partially supported by MNiSW grant no. N N201 388834.}

 \begin{abstract}
 We study the geometry of the birational map between
  an intersection of a web of quadrics in 
  $\PP_7$
 that contains a plane and the double octic branched along the discriminant of the web.
 \end{abstract}

\maketitle

\newcommand\maca{{\mathfrak a}}  
\newcommand{\AAA}{{\mathfrak A}} 
\newcommand\macB{{\mathfrak b}}  
\newcommand\macC{{\mathfrak c}}  
\newcommand\macm{{\mathfrak m}}  
\newcommand\macl{{\mathfrak l}}  
\newcommand\macn{{\mathfrak n}}  

\newcommand{\XX}{X_{16}}
\newcommand{\Pl}{\Pi}
\newcommand{\reg}{\operatorname{reg}}
\theoremstyle{remark}
\newtheorem{obs}{Observation}[section]
\newtheorem{rem}[obs]{Remark}
\newtheorem{conj}[obs]{Conjecture}
\newtheorem{examp}[obs]{Example}
\theoremstyle{definition}
\newtheorem{defi}[obs]{Definition}
\theoremstyle{plain}
\newtheorem{prop}[obs]{Proposition}
\newtheorem{theo}[obs]{Theorem}
\newtheorem{lemm}[obs]{Lemma}
\newtheorem{cor}[obs]{Corollary}
\newcommand{\ux}{\underline{x}}

\section*{Introduction}
It is a  classical fact
that there is a correspondence between
the  base locus $S$ of a net of quadrics in $\PP_5$ and the double sextic branched along the discriminant 
of the net. The latter is the moduli space of certain rank-$2$ sheaves on the former  (see \cite{m}). Moreover, if the base locus
contains a line $L$, then the two surfaces are birational. More general conditions for the existence of a birational map  were given
by Nikulin and Madonna (see \cite{nm1} and its sequels).

A precise description of the birational map between the surface $S$ and the double sextic can be found
in \cite{cynkrams}.      
In this case, $S$ is the blow-up of the double sextic along rank-$4$ quadrics in the net. The latter results from the fact that  
the map defined by the linear system $|2H - 3L - \sum_1^k L_i|$, where $H$ is the hyperplane section in $\PP_5$
and  $L_i$ are the lines on $S$  that meet $L$ (see \cite[Thm~3.3]{cynkrams}), 
is hyperelliptic.
Moreover, one can show that the birational map  factors
through another K3 surface (a space quartic that contains a twisted cubic)
  and its geometry (e.g. the contracted curves) is governed by the behaviour of the lines $L_i$.  
The birational map between the two surfaces can be also constructed
via an incidence variety (\cite{j}). 
The latter  construction was adopted in \cite{michalek} to the case of   
a generic web $\WWW = \mbox{span}(Q_0, Q_1, Q_2 , Q_3)$  in $\mathcal{O}_{\PP_{7}}(2)$,   such that its base locus $\XX$ contains a 
fixed plane $\Pl$.
More precisely, using  Bertini-type and computer algebra arguments,
Micha{\l}ek  proved that
if we put $S_8$ (resp. $X_8$) to denote the discriminant surface  of the web $\WWW$ (resp. 
the double cover of the web  $\WWW$ branched along the discriminant
surface  $S_8$) and $\WWW$ is generic enough, then the
Calabi-Yau 
varieties $\XX$ and $X_8$ are birational. However, the approach of
\cite{michalek}
gives neither explicit sufficient condition  for birationality of  $\XX$
and $X_8$ nor a method to study the geometry of the map. 

In this paper, 
for the matrices $\mq_0, \ldots, \mq_3$ that give the quadrics $Q_0, \ldots , Q_3 \in \mathcal{O}_{\PP_{7}}(2)$
 such that $Q_0 \cap \ldots \cap Q_3$ contains a plane $\Pl$
we define two auxiliary matrices $\maca$, $\AAA$ 
and use them
to obtain a surface $\mb \subset \PP_4$ and  a three-dimensional
quintic $X_5  \subset \PP_4$
that contains the surface $\mb$.  
Then, under the assumptions

\noindent
{\bf [\exactlytensingularities]:}  $\XX$ has exactly $10$ singularities on $\Pl$ and is smooth away from the plane $\Pl$, \\
{\bf [\zanurzeniebordygi]:}  no $4$  singular points of $\XX$  lie on a line, \\
{\bf  [\exactlyfortysixsingularities]:} the set  $\{ \ux \in \mb \, : \, \rank(\AAA(\ux)) \leq 2  \}$ consists of $46$ points , \\
{\bf [\isolatedsingularities]:}  the discriminant surface  $S_8$ has only isolated singularities, 

\noindent
we show that there is a birational map  $\XX \plra X_8$ that factors as the composition
$$
\XX \,  \stackrel{\sigma^{-1}}{\plra}  \,  \tx  \, \stackrel{\pi}{\lra} \, X_5  \,  \stackrel{\psi^{-1}}{\plra}  \, 
\ttx  \stackrel{\hat{\phi}}{\lra} \, X_8 \, ,
$$
where $\sigma, \psi$ are certain blow-ups, $\pi$ is resolution of the projection from $\Pl$ and
$\hat{\phi}$ is obtained via Stein factorization from restriction of the so-called Bordiga conic bundle
to the blow-up of the quintic $X_5$. In particular, under the above
assumptions $\mb$ is the so-called (smooth) Bordiga sextic.\\
Bordiga sextic and Bordiga conic bundle have been studied already by the Italian school  (see \cite{{todd}},  \cite{alzatirusso} and
the bibliography in the latter), so the above factorization enables us to give a precise description of the geometry of the birational map in question.
In particular, we are able to show that the map has no two-dimensional fibers, describe the contracted curves (Thm~\ref{thm-fibers}), classify the singularities of the discriminant of the web (and prove that all of them admit a small resolution)
and give an upper bound of their number
(see Cor.~\ref{cor-singularities}).

Our considerations yield that the assumptions [A1],$\ldots$,[A4] are fulfilled by a generic web of quadrics such that its base locus contains 
a fixed plane. Careful analysis of our arguments shows that one can assume less in order to obtain a birational map
 $\XX \plra X_8$, but 
once one omits
the above assumptions the geometry of the birational map changes.  For instance,
if [\zanurzeniebordygi] is not satisfied, the surface in $\PP_4$ one obtains as a result of the projection is no longer the Bordiga surface,
without   [\exactlytensingularities] (resp. [\exactlyfortysixsingularities]) the threefold $X_{16}$ (resp. $X_5$) has higher singularities etc. 
Still, the main strategy we use can be applied to study those degenerations - we do not follow this path in order to maintain the paper compact.  

Our motivation is twofold. First, it seems a natural question to ask under  what assumptions  a three-dimensional  Calabi-Yau analogue 
of the well-known result on K3 surfaces holds.  Second, we obtain a very precise description of a map between certain
Calabi-Yau manifolds that (with help of a computer algebra system applied to 
a given example) could be of interest on its own, for instance as a source of examples of small resolutions.

The paper is organized as follows. In Sect.~\ref{sect-singularities} we study the singularities of the threefold $\XX$
and Hodge numbers of its blow-up $\tx$. Sect.~\ref{projectionfromtheplane} is devoted to properties of projection from the plane $\Pl$.
In the next section we describe the behaviour of the restriction of Bordiga conic bundle to the blow-up of the quintic $X_5$
we defined in Sect.~\ref{projectionfromtheplane}. Finally, the last part (Sect.~\ref{discriminant}) contains
a classification of singularities of the discriminant of the web and proof of main results of the paper.

\noindent
{\em Convention:} In this note we work over the base field $\CC$. By an abuse of notation we use the same symbol to denote a homogeneous polynomial and its zero--set in projective space. 

\section{Singularities of the intersection of four quadrics and a small resolution} \label{sect-singularities}

Let $Q_0, Q_1, Q_2 , Q_3   \subset \PP_7$ be linearly independent  quadrics that contain a (fixed) plane $\Pl$ and let  
$$
\XX := Q_0 \cap Q_1 \cap Q_2 \cap Q_3
$$
be their (scheme-theoretic) intersection.

Without loss of generality we can  assume that 
$
\Pl := \{ (x_0 : \ldots : x_7) \, : \, x_0 = \ldots = x_4 = 0 \}, $
which  implies that  
 each $Q_{i}$ is given by the matrix
\[\mq_i=
\left[
\begin{array}[c]{ccccc|ccc}
   &   &                           &     &  &             &           &             \\
   &   &                           &     &  &             &           &             \\
   &   &\underline{\mq}_{i}&       &     &  &   \macB_i^{T}   &                       \\
   &   &                           &     &  &             &           &             \\
   &   &                           &     &  &             &           &             \\
\hline \rule{0mm}{5mm}    
   &    &                          &     &  &    0        &     0     &      0      \\
   &    &       \macB_i                &     &  &    0        &     0     &      0      \\
   &    &                          &     &  &    0        &     0     &      0      \\
\end{array}
\right]                    \, , 
\]
 where  $\underline{\mq_i}$ is a $5\times5$ matrix,
  $\macB_{i}:=\left[\begin{array}{c}\macl_{i}\\\macm_{i}\\\macn_{i}\end{array}\right]$ and 
$\macl_i,\, \macm_i, \, \macn_i \in \CC^5 $ are row--vectors. Moreover, in order 
to simplify our notation we put 
$\macB(y):=\sum_{i} y_{i}\macB_{i}$ 
 and 
\[\macC(x_{5},x_{6},x_{7}):=x_{5}\left[\begin{array}{cccc}\macl_{0}^{T}&
\macl_{1}^{T}&\macl_{2}^{T}&\macl_{3}^{T}\end{array} 
\right]+x_{6}\left[\begin{array}{cccc}\macm_{0}^{T}&
\macm_{1}^{T}&\macm_{2}^{T}&\macm_{3}^{T}\end{array}
\right]+x_{7}\left[\begin{array}{cccc}\macn_{0}^{T}&\macn_{1}^{T}&
\macn_{2}^{T}&\macn_{3}^{T}\end{array}
\right] \, . \]

We have (compare \cite[Prop.~1.8]{michalek})
\begin{lemm} \label{lem-singularitiesofx16}
$$
\sing(\XX) \cap \Pi = \{ (0:\ldots:x_5:x_6:x_7)  \, : \, \rank( \macC(x_{5},x_{6},x_{7}))\le3 \}
$$
In particular, if the set $\sing(\XX) \cap \Pi$ is finite, then it consists of at most $10$ points.
\end{lemm}
\begin{proof}
Observe that the intersection $\XX$
is singular at a point $x$, iff the differentials 
$dQ_{i}(x)=(\mq_{i}x)^{T}$
of quadratic forms $Q_{i}$ at $x$ are linearly dependent, that is if
there exists $(y_{0}:\dots:y_{3})\in\PP_{3}$ such that 
\[\sum_{i=0}^{3} y_{i}\mq_{i}x=0.\]
For  $x=(0:\dots:0:x_{5}:x_{6}:x_{7})\in \Pi$ the above condition reduces
to 
$\sum y_{i}(x_{5}\macl_{i}^{T}+x_{6}\macm_{i}^{T}+x_{7}\macn_{i}^{T})=0.$ \\
We can rewrite the latter as
\begin{equation} \label{eq-kernelsingularities}
\macB(y)^{T}(x_{5},x_{6},x_{7})^{T}=0.
\end{equation} 
For a fixed $y\in\PP_{3}$ there exists
a point in $\Pi$ satisfying the above relation iff $\rank( \macB(y))\le2$.
Moreover, for every $(x_5, x_6, x_7)$ and $y$ we have
\begin{equation} \label{eq-useful}
\macC(x_{5},x_{6},x_{7})y=\macB(y)^{T}(x_{5},x_{6},x_{7})^{T} \, .
\end{equation}
Therefore, 
$(0,\dots,0,x_{5},x_{6},x_{7})$ is a singularity
of $\XX$ iff there exist $y\in\PP_{3}$ such that 
$\macC(x_{5},x_{6},x_{7})y=0$ or equivalently
\[\rank( \macC(x_{5},x_{6},x_{7}))\leq 3.\] 

Finally, suppose that the set  $\sing(\XX) \cap \Pi$ is finite. Then, the number of its elements does not exceed the degree
of the determinantal variety of $4\times5$ matrices of rank $\leq 3$. The latter is $10$ by \cite[Ex.~14.4.14]{FultonInters}  (see also \cite{jozefiak}, \cite{pragacz2}).
\end{proof}

From now on we make the following {\bf assumption}: 

\vspace*{2ex}
\noindent
{\bf [\exactlytensingularities]:} {\sl $\XX$ has exactly $10$ singularities on $\Pl$ and is smooth away from the plane $\Pl$,}

\vspace*{2ex}
As an immediate consequence of  [\exactlytensingularities] we obtain  
 \begin{rem} \label{rem-rzadB} 
 For each $y \in \PP_3$ we have $ \rank(\macB(y)) \geq 2$.  Indeed, we assumed that $\XX$ has only isolated singularities on $\Pl$. Therefore, for a fixed $y \in \PP_3$, there exists
 at most one point in $\Pl$ satisfying the relation \eqref{eq-kernelsingularities}, so  $\rank(\macB(y))$ cannot be lower than $2$.
 \end{rem}
Lemma~\ref{lem-singularitiesofx16} and \cite{cheltsov} support the
following conjecture.  
\begin{conj}
a) A nodal complete  intersection of four quadrics in $\PP_{7}$ with at
most nine nodes is $\QQ$-factorial. 

\noindent b)  A nodal complete  intersection of four quadrics in $\PP_{7}$ with
exactly ten nodes that is not $\QQ$-factorial contains a plane $\Pl$.
\end{conj}

\newcommand{\PPP}{\mathcal P}  
\begin{lemm} \label{lem-quadrics} Suppose that  [\exactlytensingularities] holds.

\noindent
a)  The ideal of the set $\sing(\XX) \cap \Pi$ is generated by 
all $4 \times 4$  minors of the matrix $\macC(x_{5},x_{6},x_{7})$.
In particular, the ideal in question contains no cubics.

\noindent
b) For each   $x \in \mbox{sing}(\XX)$ there exists precisely one quadric in  $\WWW$  such that  $x$ is its singularity. 

\noindent
c) There exist three quadrics in the web $\WWW$ that meet transversally.

\noindent
d)  The set
$\{y \in \PP_3 \, : \, \rank(\macB(y)) = 2 \}$ consists of precisely $10$ points.

\end{lemm}
\begin{proof} a) Recall that  the determinantal variety  $\PP(\xs) \subset \PP_{19}$ given by the condition
\[
\rank
\left[
\begin{array}{ccc}
z_0 & \ldots & z_4 \\
\vdots &     & \vdots  \\
z_{15} & \ldots & z_{19}
\end{array}
\right] \leq 3
\]
has dimension $17$ and  degree $10$. 
Moreover, the  ideal generated by  $4 \times 4$ minors of the above matrix is perfect by \cite{ea} (see also \cite[Cor.~2.8]{bv}).
Therefore, the ring $\CC[z_0, \ldots, z_{19}]/_{\operatorname{I}(\xs)}$ is Cohen-Macaulay.
The map $(x_5,x_6,x_7) \mapsto \macC(x_{5},x_{6},x_{7})$ parametrizes a $3$-plane $\PPP \subset \CC^{20}$
 that meets $\xs$ along  ten lines. Since  the ideal $\mbox{I}(\PPP)$
 in  the ring $\CC[z_0, \ldots, z_{19}]/_{\operatorname{I}(\xs)}$
is generated by  $17$ linear forms, it
 satisfies the assumptions of \cite[Prop.~18.13]{eisenbud}.
Consequently, the quotient  $\CC[z_0, \ldots, z_{19}]/{(\mbox{I}(\xs)+\mbox{I}(\PPP))}$ is $1$-dimensional
Cohen-Macaulay and the ideal $\mbox{I}(\xs)+\mbox{I}(\PPP)$ coincides with its radical.

\noindent
 b) The plane $\PP(\PPP) \subset  \PP_{19}$
 meets the variety $\PP(\xs)$ in  exactly ten points,
 so none of the latter belongs to $\sing(\PP(\xs))$. But, as  one can check by direct computation
(see also \cite{todd}), 
 all points of $\xs$ that satisfy the condition 
\[
\rank
\left[
\begin{array}{ccc}
z_0 & \ldots & z_4 \\
\vdots &     & \vdots  \\
z_{15} & \ldots & z_{19}
\end{array}
\right] \leq 2
\]
are its  singularities. The latter implies that 
\begin{equation} \label{eq-rankonlythree}
\forall_{x \in \sing(\XX)} \quad \rank(\macC(x_{5},x_{6},x_{7})) = 3 \, .
\end{equation}
Consequently, there exists precisely one $y \in \PP_3$ that lies in the kernel of the matrix 
$\macC(x_{5},x_{6},x_{7})$. By \eqref{eq-useful}, 
the latter is equivalent to the condition $(0: \ldots: x_5 : x_6 :x_7) \in \sing(Q(y))$.
In this way we have shown the claim b).

\noindent
c) follows from b) by standard arguments. 

\noindent
d) Suppose that a point $y \in \PP_3$ satisfies  the relation 
\eqref{eq-kernelsingularities} for two various points in $\Pl$. Then, the line spanned
by both points in question lies in the kernel of the matrix $\macB(y)$ and $\rank(\macB(y)) < 2$, which is impossible by Remark~\ref{rem-rzadB}. In this way we have shown that 
 $$
 \# \{y \in \PP_3 \, : \, \rank(\macB(y)) = 2 \} \geq  \# \sing(\XX).
 $$ 
The other inequality has been shown in the proof of part b). 
\end{proof}

\begin{lemm} \label{lemm-noderecognition}
Assume that $Z_P=\{f(y_1, \ldots, y_4) = 0\}\subset \CC^4$ is a three-dimensional 
isolated hypersurface singularity that contains the  germ of the plane $\{y_1 = y_2 = 0\}$. 
If  the ideal 
\begin{equation} \label{eq-nodalideal} \nonumber
\langle \frac{\partial f}{\partial y_1}, \ldots ,  \frac{\partial f}{\partial y_4}, f, y_1, y_2\rangle \subset \mathcal O_{\CC^4,P}
\end{equation}
is maximal, then $Z_P$ is a node.
\end{lemm}
\begin{proof}
We are to show that hessian of $f$ in $P$ does not vanish.
Let $f_1, f_2 \in   O_{\CC^4,P}$ satisfy the condition $f = y_1 \cdot f_1 + y_2 \cdot f_2$. By direct computation we have
\begin{equation} \label{eq-nodalideal2} 
\langle f_1, f_2, y_1, y_2 \rangle 
=
\langle y_1, y_2, y_3, y_4 \rangle \, . 
\end{equation}
Consider the linear parts $f_{i}^{(1)} = \sum_{j=1}^{4} f_{i,j}^{(1)} y_j$
for $ i = 1,2$. 
Then hessian of $f$ in $P$ is given by
\[
\det
\left[
\begin{array}{cccc}
f_{1,1}^{(1)} & \frac{f_{1,2}^{(1)} + f_{2,1}^{(1)}}{2} &  \frac{f_{1,3}^{(1)}}{2}  &  \frac{f_{1,4}^{(1)}}{2}  \\
 \frac{f_{1,2}^{(1)} + f_{2,1}^{(1)}}{2} & f_{2,2}^{(1)}   &  \frac{f_{2,3}^{(1)}}{2}  &  \frac{f_{2,4}^{(1)}}{2} \\
 \frac{f_{1,3}^{(1)}}{2}  &  \frac{f_{2,3}^{(1)}}{2}  & 0 & 0\\
 \frac{f_{1,4}^{(1)}}{2}  &  \frac{f_{2,4}^{(1)}}{2}  & 0 & 0\\
\end{array}
\right] = 
- \det
\left[
\begin{array}{cc}
 \frac{f_{1,3}^{(1)}}{2}  &  \frac{f_{2,3}^{(1)}}{2} \\
 \frac{f_{1,4}^{(1)}}{2}  &  \frac{f_{2,4}^{(1)}}{2}
\end{array}
\right]^2 \, .
\]
To show that the right-hand side of the latter equality  does not vanish
put  $y_1 = y_2 = 0$ in \eqref{eq-nodalideal2}.
\end{proof}

\begin{lemm} \label{lem-A1nodes}
If [\exactlytensingularities] holds, then all singularities of $\XX$ are nodes (i.e. A$_1$ points).
\end{lemm}
\begin{proof}

Without loss of generality we can assume that
all singularities of $\XX$ lie in the affine chart $x_7 \neq 0$ and the
variety $Y := Q_0 \cap Q_1 \cap Q_2$ is smooth (see Lemma~\ref{lem-quadrics}). 
By abuse of notation we use the  same
symbol to denote a quadric and the dehomogenization of its equation (i.e. $x_7 = 1$).

Observe that putting $x_0=x_1=\dots=x_4=0$ in the
ideal $\langle\bigwedge^4\operatorname{Jac}(Q_0, \ldots, Q_3), Q_0,
\ldots, Q_3\rangle$
we get the ideal in $\CC[x_5,x_6]$ generated by $4\times4$ minors of
the matrix $\macC(x_5,x_6,1)$. In particular, (see Lemma~\ref{lem-singularitiesofx16}) we can compute the dimension
of the $\CC$-vector space
\begin{equation*} 
  \dim( \CC[x_0, \ldots, x_{6}]/_{\langle\bigwedge^4\operatorname{Jac}(Q_0, \ldots, Q_3), Q_0, \ldots, Q_3,x_0,\dots,x_4\rangle}) = 10.
\end{equation*}  
Moreover, the assumption 
  [\exactlytensingularities] yields an isomorphism
\begin{equation*} \label{eq-izomorfizm}
\bigoplus_{P \in \sing(\XX)} {\mathcal O}_{\CC^7,P}/_{\langle\bigwedge^4\operatorname{Jac}(Q_0, \ldots, Q_3), Q_0, \ldots, Q_3, 
x_0, \ldots, x_4\rangle{\mathcal O}_{\CC^7,P}}
 \simeq \CC[x_0, \ldots, x_{6}]/_{\langle\bigwedge^4\operatorname{Jac}(Q_0, \ldots, Q_3), Q_0, \ldots, Q_3, x_0, \ldots, x_4\rangle}
\end{equation*}
Therefore, for each $P \in  \sing(\XX)$, we have
\begin{equation} \label{eq-dim-one}
  \dim(  {\mathcal O}_{\CC^7,P}/_{\langle\bigwedge^4\operatorname{Jac}(Q_0, \ldots, Q_3), Q_0, \ldots, Q_3, 
x_0, \ldots, x_4\rangle{\mathcal O}_{\CC^7,P}}) = 1 \, .
\end{equation} 
 
Fix a point $P \in \sing(\XX)$ and assume that the germ of $Y$ near $P$  can be (analytically)  parametrized 
as the graph of a map $(x_4(x_0, \ldots, x_3),\ldots, x_6(x_0, \ldots, x_3))$. Let $\tilde{Q}_3$ be the composition 
of the above parametrization with (the dehomogenized equation of) the quadric $Q_3$. By direct computation, \eqref{eq-dim-one}
implies that the ideal
$$
\langle\tilde{Q}_3, \frac{\partial{\tilde{Q}_3}}{\partial x_0}, \ldots, 
 \frac{\partial{\tilde{Q}_3}}{\partial x_3}\rangle + \mbox{I}(\Pl) \subset {\mathcal O}_{Y,P}
$$
is maximal. By Lemma~\ref{lemm-noderecognition} the point $P$ is an A$_1$ singularity of $\XX$.
\end{proof}

We introduce the following notation:
\begin{equation} \label{eq-rozdmuchanieplaszczyzny}
\sigma: \tx \rightarrow \XX
\end{equation}
is   the blow-up of $\XX$ along the plane $\Pl$
and $\ts$ stands for
the {\sl strict transform } of the plane $\Pl$ under the blow-up $\sigma$.
The variety $\tx$ is smooth and the blow-up in question replaces the $10$ nodes with $10$ disjoint 
smooth rational curves 
\begin{equation} \label{eq-tenlines}
E_{1},\dots,E_{10} \subset \ts.
\end{equation}

\vspace*{2ex}
\noindent
{\bf Convention:}  {\sl In the sequel, we shall identify smooth points of $\XX$
with their images in $\tx$, i.e. write $P$ instead of $\sigma(P)$ whenever it leads to no ambiguity.}

In the next section we will use the following lemma.
\begin{lemm} \label{lem-hn}
The variety  $\tx$ is a projective Calabi--Yau
manifold with 
the following Hodge diamond
\[
\begin{array}{ccccccc}
&&& 1 &&&\\
&& 0 && 0 &&\\
& 0 && 2 && 0 &\\
1 && 56 && 56 && 1\\
& 0 && 2 && 0 &\\
&& 0 && 0 &&\\
&&& 1 &&&
\end{array}
\]
\end{lemm}
\begin{proof}
By Lemma~\ref{lem-quadrics}.b we can 
assume that $Y=Q_{0}\cap Q_{1}\cap Q_{2}$ is smooth.  
Let $\sigma:\tilde Y\lra Y$ be the blow--up of $Y$ along $\Pl$ with
exceptional divisor $E$.
We have
\begin{eqnarray*}
&&  \sigma_{*}\moty(kE)=\mo_{Y}, \text{ for } k\ge0,\\
&&  R^{1}\sigma_{*}\moty(E)=0,\\
&&  R^{1}\sigma_{*}\moty(2E)=\mo_{\Pl}(-1).
\end{eqnarray*}
Since $\moty(\tx)=\sigma^{*}\mo_{Y}(X)\otimes\moty(-E)$ using the
projection formula we get 
\begin{eqnarray*}
&&  \sigma_{*}\moty(-k\tx)=\mo_{Y}(-kX), \text{ for } k\ge0,\\
&&  R^{1}\sigma_{*}\moty(-\tx)=0,\\
&&  R^{1}\sigma_{*}\moty(-2\tx)=\mo_{\Pl}(-5).
\end{eqnarray*}
The Leray spectral sequence and the Kodaira vanishing imply 
\[H^{i}(\mo_{\ty}(-\tx))=0\text{ for }i\le3,\quad 
H^{4}(\mo_{\ty}(-\tx))\cong \mathbb C.\]
Since 
\begin{eqnarray*}
&&  H^{i}(\mo_{Y}(-2Y))=0, \text{ for }i\le3,\\
&&  H^{4}(\mo_{Y}(-2X))\cong H^{0}(\mo_{Y}(2))\cong \mathbb C^{33},\\
&&  H^{4}(\mo_{\ty}(-2\tx))\cong H^{0}(\mo_{\ty}(\tx))
  \cong H^{0}(\mo_{Y}(X)\otimes I(\Pl))\cong \mathbb C^{27},\\
&&  H^{i}(R^{1}\sigma_{*}(\mo_{\ty}(-2\tx)))=0,\text{ for }i=0,1\\
&&  H^{2}(R^{1}\sigma_{*}(\mo_{\ty}(-2\tx)))\cong
  H^{2}(\mo_{\Pl}(-5))\cong \mathbb C^{6}
\end{eqnarray*}
the Leray spectral sequence implies 
\[H^{i}(\mo_{\ty}(-2\tx))=0, \text{ for }i\le3\]
and consequently
\[H^{i}(\mn^{\vee}_{\tx|\ty})=0\text { for }i\le2.\]
From the exact sequence 
\[0\lra\sigma^{*}\Omega^{1}_{Y}\lra\Omega^{1}_{\ty}\lra\Omega^{1}_{E/\Pl}\lra0\]
we get 
\[\sigma_{*}\Omega^{1}_{\ty}=\Omega^{1}_{Y},\quad
R^{1}\sigma_{*}\Omega^{1}_{\ty}=\mo_{\Pl}\]
and so
\[H^{1}\Omega^{1}_{\ty}\cong\mathbb C^{2}.\]
Similarly, the exact sequence
\[0\lra\sigma^{*}(\Omega^{1}_{Y}(-X))\otimes \mo_{\ty}(E)
\lra\Omega^{1}_{\ty}(-\tx)\lra\Omega^{1}_{E/\Pl}(-1)
\otimes\sigma^{*}\mo_{Y}(-X)\lra0\] 
implies 
\[\sigma_{*}\Omega^{1}_{\ty}(-\tx)\cong\Omega^{1}_{Y}(-X)
\text{ and }R^{1}\sigma_{*}\Omega^{1}_{\ty}(-\tx)\cong\mn_{\Pl|Y}
\otimes\mo_{Y}(-X).\]
Twisting the exact sequence 
\[0\lra\mn_{\Pl|Y}\lra\mn_{\Pl|\PP^{7}}\lra\mn_{Y|\PP^{7}}|\Pl\lra0\]
with $\mo_{Y}(-X)\cong\mo_{Y}(-2)$ we get
\[H^{0}\mn_{\Pl|Y}\otimes\mo_{Y}(-X)=H^{0}\mn_{\Pl|Y}
\otimes\mo_{Y}(-X)=0\qquad\text{and}\qquad 
H^{1}\mn_{\Pl|Y}\otimes\mo_{Y}(-X).\]
Since $H^{3}(\Omega^{1}_{Y}(-X))\cong H^{1}(\mathcal T_{Y})=36$, while 
$H^{3}(\Omega^{1}_{\ty}(-\tx))\cong H^{1}(\mathcal T_{\ty})=33$
the Leray spectral sequence yields
\[H^{i}\Omega^{1}_{\ty}(-\tx)=0, \text { for }i=0,1,2.\]
From the exact sequence 
\[0\lra \Omega^{1}_{\ty}(-\tx)\lra \Omega^{1}_{\ty}\lra
\Omega^{1}_{\ty}\otimes\mo_{\tx}\lra0\] 
we conclude
\[H^{1}(\Omega^{1}_{\ty}\otimes
\mo_{\tx})\cong H^{1}\Omega^{1}_{\ty}\cong \mathbb C^{2}.\]

Finally, the exact sequence
\[0\lra\mn^{\vee}_{\tx|\ty}\lra\Omega^{1}_{\ty}\otimes
\mo_{\tx}\lra\Omega^{1}_{\tx}\lra0\]
yields 
\[H^{1}\Omega^{1}_{\tx}\cong H^{1}(\Omega^{1}_{\ty}\otimes
\mo_{\tx})\cong \mathbb C^{2}.\]
The standard computation with help of \cite[Example~3.2.12]{FultonInters} yields that 
 the Euler number $e(\tx)=-108$ (see also \cite[Prop.~1.14]{michalek}), so we can compute $h^{1,2}(\tx)$.
\end{proof}

As another consequence of [\exactlytensingularities]  we obtain the following simple observation.
\begin{rem}  \label{rem-noquadricsofrankfour}
The web $\WWW$ contains no rank-$4$ quadrics.
\end{rem}
\begin{proof}
Suppose that $Q_0 \in \WWW$ is a rank-$4$ quadric. Then it is a cone through the  $3$-space 
$\sing(Q_0)$ 
over a smooth quadric in $\PP_3$. The latter contains no planes, so 
the $3$-space $\sing(Q_0)$ and the plane $\Pl$ meet.
On the other hand,
since each point in $\sing(Q_0) \cap Q_1 \cap Q_2  \cap Q_3$ is a singularity of $\XX$,
the assumption
[\exactlytensingularities] implies that  $\sing(Q_0)$ meets $\Pl$ in exactly one point   
$P \in \sing(\XX)$. Moreover, we have 
$\sing(Q_0) \cap Q_1 \cap Q_2 \cap Q_3 = \{ P \}$. \\
 Lemma~\ref{lem-quadrics}.b yields 
that the quadrics $Q_1$, $Q_2$, $Q_3$ are smooth in $P$.
By B\'ezout the intersection multiplicity of $\sing(Q_0)$, $Q_1$, $Q_2$, $Q_3$ 
in the point $P$ is $8$. The latter exceeds the product of multiplicities of the varieties
in question in the point $P$. From \cite[Thm~6.3]{draper} we obtain the inequality:  
\begin{equation} \label{eq-1wektor}
\dim(\sing(Q_0) \cap \mbox{T}_{P}Q_1 \cap \mbox{T}_{P}Q_2 \cap \mbox{T}_{P}Q_3) \geq 1 \, .
\end{equation}
To complete the proof, suppose that $\sing(Q_0)$ is the zero set of the 
coordinates $x_0$, $x_1$, $x_6$, $x_7$. Recall that  $\Pl$ is given by vanishing of $x_0$, \ldots, $x_4$, so we have  
$P = (0: \ldots:1:0:0)$ and only $12$ entries in the matrix $\mq_0$ do not vanish.

The point $P$ is a node on $\XX$, so $\dim(\mbox{T}_{P}Q_1 \cap \mbox{T}_{P}Q_2 \cap \mbox{T}_{P}Q_3)= 4$.
Consider the affine chart $x_5 = 1$. The inequality~\eqref{eq-1wektor} implies that there exists a nonzero 
$v := (0, 0, v_2, v_3, v_4, 0 , 0)$ in the $4$-dimensional intersection of the tangent spaces. 
Furthermore, all quadrics in question contain $\Pl$, so the $4$-space 
contains the vectors $(0, \ldots, 0, 1, 0)$ and  $(0, \ldots, 0, 1)$.    
Consequently, a parametrization of 
$\mbox{T}_{P}Q_1 \cap \mbox{T}_{P}Q_2 \cap \mbox{T}_{P}Q_3$ is given by the map
$$
(\lambda_1, \ldots, \lambda_4) \mapsto \lambda_1 v + \lambda_2 w + \lambda_3  (0, \ldots, 1, 0) +  \lambda_3  (0, \ldots, 1),
$$  
where $w := (w_0, \ldots, w_4, 0, 0)$.

Finally, direct computation shows that  
intersection of the tangent cones  $\mbox{C}_{P}Q_0$, $\mbox{T}_{P}Q_1$, $\mbox{T}_{P}Q_2$, $\mbox{T}_{P}Q_3$
consists of two planes. The latter is impossible because we assumed  the point $P$ to be  a node of $\XX$. Contradiction.
\end{proof}

\section{Projection from the  plane} \label{projectionfromtheplane}

Here we maintain the notation of the previous section. Moreover,
we assume that  [\exactlytensingularities] holds and

\vspace*{2ex}
\noindent
{\bf [\zanurzeniebordygi]:} {\sl no $4$  singular points of $\XX$  lie on a line.}

\vspace*{2ex}

In view of Lemma~\ref{lem-quadrics}.a it seems natural to ask whether the assumption 
 [\exactlytensingularities] implies  [\zanurzeniebordygi].
The example below shows that this is not the case. 
\begin{examp}
  Consider the following $8\times8$ symmetric matrices
{\small
\[\mq_{0} :=\left[
\begin{array}{*{8}{r}}
0 & -4 & 4 & 0 & 1 & -2 & 0 & 1 \\
-4 & 4 & 4 & 3 & -3 & 2 & 2 & -2 \\
4 & 4 & 4 & 1 & -1 & 0 & -1 & 0 \\
0 & 3 & 1 & -2 & -1 & -2 & -1 & 2 \\
1 & -3 & -1 & -1 & 2 & 0 & 0 & 0 \\
-2 & 2 & 0 & -2 & 0 & 0 & 0 & 0 \\
0 & 2 & -1 & -1 & 0 & 0 & 0 & 0 \\
1 & -2 & 0 & 2 & 0 & 0 & 0 & 0
\end{array}
\right]
\quad \mq_{1} :=\left[
\begin{array}{*{8}{r}}
-2 & 2 & -1 & -3 & 0 & 0 & 0 & -2 \\
2 & 0 & -4 & 1 & 1 & 4 & -3 & 2 \\
-1 & -4 & 2 & 3 & 1 & 1 & 0 & -1 \\
-3 & 1 & 3 & -2 & -3 & 1 & -3 & 1 \\
0 & 1 & 1 & -3 & 2 & -2 & 0 & 0 \\
0 & 4 & 1 & 1 & -2 & 0 & 0 & 0 \\
0 & -3 & 0 & -3 & 0 & 0 & 0 & 0 \\
-2 & 2 & -1 & 1 & 0 & 0 & 0 & 0
\end{array}\right]
\]
\[
\mq_{2} :=\left[
\begin{array}{*{8}{r}}
-4 & 1 & 1 & 1 & -2 & -1 & -1 & -1 \\
1 & 4 & -1 & -1 & -3 & -3 & 0 & 1 \\
1 & -1 & 2 & -4 & 0 & 2 & 2 & 1 \\
1 & -1 & -4 & 2 & -1 & -1 & 1 & 1 \\
-2 & -3 & 0 & -1 & -4 & -2 & 0 & 0 \\
-1 & -3 & 2 & -1 & -2 & 0 & 0 & 0 \\
-1 & 0 & 2 & 1 & 0 & 0 & 0 & 0 \\
-1 & 1 & 1 & 1 & 0 & 0 & 0 & 0
\end{array}
\right]\quad
\mq_{3} :=
\left[
\begin{array}{*{8}{r}}
-4 & -1 & -4 & 3 & -1 & \phantom{-}4 & 1 & 0 \\
-1 & 4 & -4 & -3 & 0 & 3 & -1 & 0 \\
-4 & -4 & 0 & 1 & 0 & 1 & 1 & 1 \\
3 & -3 & 1 & 2 & 2 & 1 & 0 & -2 \\
-1 & 0 & 0 & 2 & 4 & 3 & 0 & 0 \\
4 & 3 & 1 & 1 & 3 & 0 & 0 & 0 \\
1 & -1 & 1 & 0 & 0 & 0 & 0 & 0 \\
0 & 0 & 1 & -2 & 0 & 0 & 0 & 0
\end{array}\right]
\]}
By direct computation with help of \cite{GPS01}, the intersection  in $\PP_7$ of
the quadrics defined by the above matrices has $10$ isolated
singularities on the  plane $\Pl$ and is smooth elsewhere. In the same way one checks that  
$4$ singular points of the intersection in question lie on the line $(0:\ldots:0:x_6:x_7)$ and  are
given by the equation 
$$
19 x_6^4 + 102 x_6^3x_7 + 189 x_6^2x_7^2 + 137 x_6x_7^3 + 27 x_7^4 = 0.
$$ 
\end{examp}

In this section we study the projection
$
\XX \setminus \Pl  \ni (x_0 : \ldots: x_7) \mapsto (x_0 : \ldots: x_4) \in \PP_4
$
from the plane $\Pi$. Observe that the map in question lifts to a regular map 
\begin{equation} \label{eq-podniesionaprojekcja}
\pi:\tx\lra \PP_{4}
\end{equation}
given by the linear system 
$|H-\ts|$, where $H$ is the pullback of a hyperplane section under the blow-up  $\sigma: \tx \rightarrow \XX$, and $\ts$
stands for the strict transform of $\Pi$. 

\begin{lemm} \label{lem-intersectionnumbersX5}
We have the following intersection numbers: 
  \begin{eqnarray*}
    &&H^{3}=16,\\
    &&H^{2} \cdot \ts=1,\\
    &&H \cdot \ts^{2}=-3,\\
    &&\ts^{3}=-1,\\
    &&(H-\ts)^{3}=5.
  \end{eqnarray*}
\end{lemm}
\begin{proof}
  The first two statements are obvious. 
The intersection number
  $H \cdot \ts^{2}$ equals the intersection number in $\ts$ of the
  restrictions  $H|_\ts$, $\ts|_\ts$. Since $\ts$ is a blow--up of the
  plane $\Pl$ in 10 points, the restriction $H|_\ts$ is the pullback $l$
  of a line in $\Pl$. Moreover, $\ts|_\ts$ is the normal bundle of $\ts$ in
  the Calabi--Yau manifold $\tx$. Hence it is the canonical divisor
  $K_{\ts}=-3l+\sum_{1}^{10}E_{i},$
where $E_{1}, \ldots, E_{10}$ are the 10  exceptional curves (see \eqref{eq-tenlines}).
Finally, we have  
$$
H \cdot \ts^{2}=(l \cdot (-3l + \sum_{1}^{10} E_{i}))_{\ts} = -3.
$$
Similarly, $\ts^{3}=((-3l + \sum_{1}^{10} E_{i})^{2})_{\ts}=9-10=-1$. The last statement follows from 
  Newton's formula. 
\end{proof}

To simplify our notation we put 
$\ux :=  (x_0 : \ldots :x_4)\in\mathbb P^{4}$
and define the following matrices :
\begin{equation} \label{eq-matrices}
\maca(\ux):=\left[
\begin{array}[c]{cccc}
  \macl_0 \ux   &        \macl_1 \ux   &    \macl_2 \ux   &    \macl_3 \ux      \\
  \macm_0 \ux   &        \macm_1 \ux   &    \macm_2 \ux   &    \macm_3 \ux      \\
  \macn_0 \ux   &        \macn_1 \ux   &    \macn_2 \ux   &    \macn_3 \ux   \\
\end{array}
\right] , 
\quad
\AAA(\ux):=
 \left[
\begin{array}{cc}
\ux^{T} \underline{\mq}_{0} \ux &   \\
\ux^{T} \underline{\mq}_{1} \ux &   \\
\ux^{T} \underline{\mq}_{2} \ux &   \\
\ux^{T} \underline{\mq}_{3} \ux &  
\end{array}
\begin{array}{c}
   \\
\raisebox{3ex}{${\maca(\ux)}^{T}$}   \\
\end{array}
\begin{array}{c}
    \\
   \\
    \\
  \end{array}
\right] \, .
\end{equation}

Observe that the following equality holds
(cf. \cite[p.~30]{alzatirusso})
\begin{equation} \label{eq-relacjaAB}
\maca(\ux)y=\macB(y)\ux.
\end{equation}

Let $\underline Q_i$ be the quadratic form associated to the matrix $\underline{\mq_i}$ and 
let  $\mc_i$ denote the cubic given by the degree-3 minor of the matrix $\maca(\ux)$ obtained by deleting its 
$i$-th column,  e.g. 
\[
\mc_0 :=   \det \left[
\begin{array}[c]{ccc}
    \macl_1 \ux   &    \macl_2 \ux   &    \macl_3 \ux      \\
    \macm_1 \ux   &    \macm_2 \ux   &    \macm_3 \ux      \\
    \macn_1 \ux   &    \macn_2 \ux   &    \macn_3 \ux    \\
\end{array}
\right].
\]

\begin{lemm} \label{lem-equationofquintic}
a) The image of $\tx$ under $\pi$ is the quintic $X_5$ given by the equation
\begin{equation}   \label{eq-quintic}
\det(\AAA(\ux)) = 
\mc_0 \cdot \underline{Q}_0 - 
\mc_1  \cdot \underline{Q}_1 +
\mc_2  \cdot \underline{Q}_2 -
\mc_3  \cdot \underline{Q}_3
 = 0 \, .
\end{equation}

\noindent
b)  The image of $\ts$ under $\pi$ is the (smooth) Bordiga sextic $\mb \subset \PP_4$
 given by vanishing of the cubics $\mc_0, \ldots, \mc_4$ (i.e. all   $3\times3$  
minors of the matrix $\maca(\ux)$). Moreover, the map $\pi|_{\ts}: \ts \rightarrow \mb$ is 
an isomorphism.
\end{lemm}

\begin{proof}
Obviously,
the restriction of the quadric  $\sum_0^3 \alpha_i Q_i$
to the $3$-space
$$
\operatorname{span}\{\ux,\Pi\} = \{ (\mu_0
x_0:\ldots:\mu_0 x_3:\mu_0 x_4:\mu_1:\mu_2:\mu_3) \, | \, 
(\mu_0:\mu_1:\mu_2:\mu_3) \in \PP_3 \} 
$$
is given by the polynomial
\begin{equation} \label{eq-restriction}
(\sum_0^3 \alpha_i \ux^{T}\underline{\mq}_{i}\ux)\mu_0^2 
+2(\sum_0^3\alpha_i (\macl_i\ux))\mu_0 \mu_1 +2(\sum_0^3\alpha_i (\macm_i\ux))\mu_0 \mu_2 +2(\sum_0^3\alpha_i (\macn_i\ux))\mu_0 \mu_3  .
\end{equation}

\noindent
a) Observe that
$\ux \in \mathbb P_{4} \setminus \pi(\ts)$ lies in the image of $\XX$ under the projection from $\Pl$ iff
the planes residual to $\Pl$ in the intersections of the quadrics $Q_i$
with the $3$-space $\operatorname{span}\{\ux,\Pl\}$ 
intersect. By \eqref{eq-restriction}, the latter is equivalent to the vanishing
$\det(\AAA(\ux)) = 0$. 
Laplace formula completes the proof.

\noindent
b) From \eqref{eq-restriction} we obtain that the condition 
$$
\sum_0^3\alpha_i (\macl_i\ux)=\sum_0^3\alpha_i (\macm_i\ux) = \sum_0^3\alpha_i (\macn_i\ux) = 0
$$
is satisfied  iff the restriction  $(\sum_0^3 \alpha_i Q_i)|_{\operatorname{span}\{\ux,\Pi\}}$ is the double plane $2\Pi$.
The latter holds precisely when $\ux$ lies in the image of $\Pi$ under the projection in question. 

It is well known that, for a generic  $4\times3$ matrix whose entries are linear forms in five variables,
 the surface given by the vanishing of $3\times3$ minors is  $\PP_2$ blown-up in 10 points (see e.g. \cite{alzatirusso}).
Still, it is not always the case (see e.g. \cite{todd}). To see that our surface is indeed the (smooth) Bordiga sextic, observe
that 
 the linear system $|H-\ts|$ restricts on $\ts$ to the
complete linear system $|4l-\sum_{i=1}^{i=10} E_{i}|$.
 We apply  \cite[Lemma~2.9.1]{braun}
 to show that the system in question embeds $\ts$ into $\PP_{4}$
as the (smooth) Bordiga sextic.
By Lemma~\ref{lem-quadrics}.a no cubic contains all singularities of $\XX$. Suppose 
that $8$ singularities of $\XX$ lie on a conic. Then its product with the line through the remaining two singular points is a cubic containing   $\sing(\XX)$.
Consequently the existence of such a conic is ruled out by Lemma~\ref{lem-quadrics}.a.
 Finally no $4$ singularities lie on a line by the assumption [\zanurzeniebordygi].
\end{proof}

\begin{rem} \label{rem-rankdwa}
a) Observe that, since the (scheme--theoretic) intersection $\mb$ of the zeroes of the degree-$3$ minors of the matrix $\maca(\ux)$
is smooth, we have
$$
\mbox{rank}(\maca(\ux)) = 2 \quad \mbox{ for every } \ux \in \mb. 
$$

\noindent
b) The rational curves  $E_{1},\dots,E_{10}\subset \tx$ are mapped by $\pi$
to lines in $\PP_{4}$ contained in the Bordiga sextic. Indeed,  we have
$(H-\ts) \cdot E_{j}=((4l-\sum E_{i}) \cdot E_{j})_{\ts}=1$  for
$j=1,\dots,10$.\\
Geometrically,
points on such a line
 $\subset \mb$ 
correspond to the $3$-spaces in the $4$-space $\mbox{T}_P\XX$, where
$P$ is a  node of $\XX$, that contain the plane $\Pl$. 

\end{rem}
We introduce the following notation:
$$
U := \tx \setminus (\ts \cup \bigcup_{V \operatorname{ linear,  } V \subset \XX, V \cap \Pl \neq \emptyset}  \sigma^{-1}(V)) \, .
$$

\begin{lemm} \label{lem-X5isnormal}
Suppose that  [\exactlytensingularities], [\zanurzeniebordygi] hold. \\
a) The map $\pi|_U$ is an isomorphism onto the image  and we have the equality
$ \pi(U) = (X_5 \setminus \mb).$ \\
b) The inclusion $\sing(X_5) \subsetneq  \mb$ holds. In particular,
the quintic $X_5$ is normal. 
\end{lemm}
\begin{proof}
a) Fix $P \in U$. Then $\sigma(P) \notin \Pl$.  Since $X_{16}$ is an intersection of quadrics
 we have the equality
$$
 \mbox{span}(\sigma(P), \Pl) \cap X_{16} = \Pl \cup \{ \sigma(P) \},  \mbox{ where }  \sigma(P) \notin \Pl
$$
which implies that $\pi|_U$ is injective and the linear map  $d_{P}\pi$ is an isomorphism.

We claim that 
\begin{equation*} \label{eq-inclusion}
\pi (\tx \setminus U) = \mb.
\end{equation*}
Let $V \subset X_{16}$,  $V \nsubseteq \Pl$ be a linear subspace such that $V \cap \Pl \neq \emptyset$.
Let $\sigma(P_1) \in (V \setminus \Pl)$ and  let $\sigma(P_2) \in (V \cap \Pl)$. By definition of $\pi$ 
all points from  $\mbox{span}(\sigma(P_1), \sigma(P_2))  \setminus \{\sigma(P_2)\}$ lie in one fiber of $\pi$. On the other hand,
the proper transform of the line  $\mbox{span}(\sigma(P_1), \sigma(P_2))$ under $\sigma$ meets $S$.
Since $\pi$ maps that proper transform of the line in question to one point and
$\pi(P_2) \in \mb$
 we have $\pi(P_1) \in \mb$, and 
we obtain  the claim.

It remains to  show the inclusion  
\begin{equation*} \label{eq-inclusion2}
 \pi(U) \subset  (X_5 \setminus \mb).
\end{equation*}
Suppose that 
 $\pi(P_3) = \pi(P_4)$, where $P_3\in \tx \setminus U$ and $P_4 \in U$. If $\sigma(P_3) \in \reg(\XX)$, then
the line $\mbox{span}(\sigma(P_3),\sigma(P_4))$ is tangent to $\XX$ in $\sigma(P_3)$ and meets it in   $\sigma(P_4)$.
In particular, it is contained in each quadric of the system  $\WWW$, so  
$\mbox{span}(\sigma(P_3),\sigma(P_4)) \subset  \XX$ and  $P_4 \notin U$. Contradiction. \\
Similar argument yields contradiction when  $\sigma(P_3) \in \sing(\XX)$.

\noindent
b) By [\exactlytensingularities] and part a) we know that  $\sing(X_5) \subset  \mb$. Suppose that 
 $\sing(X_5) = \mb$. Since $\mb$ is smooth, Lemma~\ref{lem-equationofquintic}.a implies that
$\det(\AAA(\ux)) \in \mbox{I}(\mb)^2$. The latter is impossible because the ideal  $\mbox{I}(\mb)$ is generated by the cubics 
$\mc_0,  \mc_1,  \mc_2, \mc_3$. \\
Finally $X_5$ is a $3$-dimensional hypersurface with at most $1$-dimensional singularities, so it is normal.
\end{proof}
After those preparations we can study  higher-dimensional fibers of $\pi$.
\begin{lemm} \label{lem-gluedpoints}
a) The map $\pi$ has no two-dimensional fibers and its only one-dimensional fibers are
proper transforms of  lines on $\XX$ that meet $\Pl$ but are not contained in $\Pl$. 

\noindent
b)  The following equality  holds
\begin{equation} \label{eq-equationofsingularities}
\sing(X_5) := \{ \ux \in \mb \, : \, \rank(\AAA(\ux)) \leq 2  \} \, .
\end{equation}

\noindent
c) The map $\pi$ has only finitely many 
one--dimensional fibers.
 \end{lemm}
\begin{proof}
a) As we have already shown in the proof of Lemma~\ref{lem-X5isnormal}
the proper transform of each line on $\XX$ that meets $\Pl$ but is not contained in $\Pl$ lies in a fiber of $\pi$.

The regular map $\pi$ is birational and its image is normal, so we can apply Zariski's Main Theorem
 \cite[Thm~5.2]{hs} to see that the map
 $\pi$ has connected fibers. 
Moreover, by Lemma~\ref{lem-equationofquintic}.b 
\begin{equation} \label{eq-onepointonplane}
\mbox{ each fiber of } \pi \mbox{ meets the surface } S  \mbox{ in at most one point.}
\end{equation}
Let $F$ be a fiber of $\pi$ such that $\dim(F) \geq 1$. Let $P_{1},P_{2} \in (F \setminus S)$. Then the $3$-spaces
 $\mbox{span}(\sigma(P_1), \Pi)$,   $\mbox{span}(\sigma(P_2), \Pi)$ coincide, so the
 line  $\mbox{span}(\sigma(P_1), \sigma(P_2))$
meets the plane $\Pi$. Obviously, the intersection point does not coincide
 with $P_1$, $P_2$. Since $\XX$ is intersection of quadrics,
we have $\mbox{span}(\sigma(P_1), \sigma(P_2)) \subset \XX$, which implies that
\begin{equation*} \label{eq-lineinfiber}
\mbox{span}(\sigma(P_1), \sigma(P_2)) \subset \sigma(F)  \, .  
\end{equation*}

Suppose that the fiber $F$ contains a point $P_3 \notin S$ such that $\sigma(P_3) \notin \mbox{span}(\sigma(P_1), \sigma(P_2))$.
Then, arguing as in \eqref{eq-lineinfiber}, we show that $\mbox{span}(\sigma(P_1), \sigma(P_3))$ is a line contained in
$\sigma(F)$ and meeting the plane $\Pi$. But, \eqref{eq-onepointonplane} implies that
 the proper transforms (under the blow-up $\sigma$)
of two lines meeting $\Pi$ in different points 
cannot lie in the same fiber of $\pi$. Consequently, by \eqref{eq-onepointonplane}, 
the image $\sigma(F)$ is a plane in 
$\XX$ that intersects $\Pl$ in precisely one point. Observe that the
planes $\sigma(F)$, $\Pi$ meet in a singularity of $\XX$. Let $H$ be the pullback of a 
hyperplane section under the blow-up $\sigma$ and let $\widetilde{\sigma(F)}$ denote the proper transform  
of  $\sigma(F)$. If we put   $\tilde{l}$ (resp.  $\tilde{m}$) to denote the proper transform of 
a line in $\sigma(F)$ (resp. in $\Pi$)  that  runs through no singularities of $\XX$, then
we obtain the following table of
 intersection numbers.
\renewcommand\arraystretch{1.2}
\begin{equation*} \label{eq-mp1}
\begin{array}{c|c|c|c}
           & \widetilde{\sigma(F)} & S   & H        \\  \hline  
\tilde{l}  &     -3             & 0   &  1              \\  \hline  
\tilde{m}  &      0             & -3  &  1               \\  \hline  
H^2  &      1             &  1  &  16   
\end{array}
\end{equation*}
The resulting matrix has non-zero determinant, so Picard number of $\tx$ is at least $3$,
which is impossible by Lemma~\ref{lem-hn}. This contradiction shows that the fiber $F$ coincides 
with the proper transform of the line 
$\mbox{span}(\sigma(P_1), \sigma(P_2))$.

\noindent
b) As in the proof of Lemma~\ref{lem-equationofquintic}, we see that  the line through
 the points $(\ux,x_{5},x_{6},x_{7})$ and
$(0,x_{5}',x_{6}',x_{7}')$ is contained in $\XX$ iff for any $\lambda \in \CC$ and $i = 0, \ldots, 3$ we have
\[\ux^{T}\underline{\mq}_{i}\ux + 2(\macl_{i}\ux, \macm_{i}\ux, \macn_{i}\ux)(x_{5},x_{6},x_{7})^{T}   +2\lambda
(\macl_{i}\ux, \macm_{i}\ux, \macn_{i}\ux)(x_{5}',x_{6}',x_{7}')^{T}=0.
\]

Fix $\ux \in \mb$.
From Remark~\ref{rem-rankdwa}.a we know that  $\rank(\maca(\ux)) = 2$.
Consequently, there exist points
$(x_{5},x_{6},x_{7})$ and $(x_{5}',x_{6}',x_{7}')$ such that the line
spanned by $(\ux,x_{5},x_{6},x_{7})$ and 
$(0,x_{5}',x_{6}',x_{7}')$ is contained in $\XX$ if and only if
$\rank(\AAA(\ux)) = 2$. 

\noindent
c) Assume to the contrary that the map $\pi$ contracts infinitely many 
lines. Then there is a ruled surface $ G \subset \tilde X_{16}$ such that
the fibers of $G$ are contracted by $\pi$. Let $l$ (resp. $E_i$)
be the class of a (general) fiber of $G$, (resp. of an exceptional curve of the blow-up $\sigma$). 
We have the following intersection numbers
\begin{equation} \label{table-intnumfinite}
\begin{array}{c|c|c|c}
  &S&G&H\\\hline
l&1&-2&1\\\hline
E_{i}&-1&\nu&0
\end{array}
 \end{equation}
The above table yields immediately that $H$ and $S$ are linearly independent in 
$\mbox{Pic}(\tilde X_{16}) \otimes \QQ$.
Since $h^{1,1}(\tilde X_{16})=2$, we can find $d_H, d_S \in \QQ$ such that
$G \sim_{\operatorname{num}} d_H H +  d_S S$. From \eqref{table-intnumfinite}
we obtain  
$$
G \sim_{\operatorname{num}} (\nu - 2)  H -  \nu S.
$$
Therefore Lemma~\ref{lem-intersectionnumbersX5} yields the equality
\[(H-S)^2 \cdot G= 5 \nu - 22.\]
As the divisor $G$ is contracted by $\pi$ we conclude that
 $\nu=\frac{22}{5}$, which is impossible by  \eqref{table-intnumfinite}. 
\end{proof}

In particular, Lemma~\ref{lem-gluedpoints} implies  that the map $\pi:\tx\lra X_{5}$ is 
a resolution of singularities of the quintic  $X_{5}$.
As $\pi$ contracts only finitely many curves (i.e. the singular locus
of $X_5$ is zero-dimensional), 
it  is in fact a small resolution that introduces exactly one copy of
$\PP_1$ over each singularity.

The lemma below gives a simple
criterion when the quintic $X_5$ is nodal. 

\begin{lemm} \label{nodesofquintic} All singularities of 
 the quintic $X_{5}$ are  nodes iff the set $\sing(X_5)$ consists of $46$ points. 
\end{lemm}
\begin{proof}
Let $\mu(\cdot)$ stand for the Milnor number.
 Lemma~\ref{lem-X5isnormal} yields that 
the regular map $\pi:\tx\lra X_{5}$ is birational. 
By Lemma~\ref{lem-gluedpoints}
it contracts only
the lines in $\XX$ that intersect the plane $\Pl$. The contracted lines are
pairewise disjoint,
 so we obtain
\begin{equation*} \label{eq-milnor-1}
-108-\#(\sing(X_{5})) = e(X_{5}) = -200+\sum_{P \in \sing(X_{5})} \mu(P, X_{5}),
\end{equation*}
where the second equality results from \cite[Cor.~5.4.4]{dimca}.
To complete the proof recall that 
the Milnor number of a singularity is $1$ iff the singularity in question is an  A$_1$ point.
\end{proof}

\section{Restriction of the  Bordiga conic bundle} \label{hyperellipticmap}

In this section we maintain the assumptions and notation of the previous one, i.e. we assume that 
[\exactlytensingularities], [\zanurzeniebordygi] hold.
 In particular, 
{\sl the scheme-theoretic intersection of the zeroes of the degree-$3$ minors of the matrix $\maca(\ux)$  is smooth} 
 (see \eqref{eq-matrices})
and {\sl the locus $\{y \in \PP_4 \, : \, \rank(\macB(y)) = 2 \}$ consists of $10$ points}. 
Moreover,  we make the following {\bf assumption}: 

\vspace*{1ex}
\noindent
{\bf [\exactlyfortysixsingularities]:} {\sl the set  $\{ \ux \in \mb \, : \, \rank(\AAA(\ux)) \leq 2  \}$ consists of $46$ points }.

\vspace*{1ex}
One can show (see \cite[Ex.~3 on p.~35]{alzatirusso}) that the rational map 
\begin{equation} \label{eq-bordigarational}
 \PP_4 \setminus \mb \ni \ux \mapsto (\mc_0(\ux):-\mc_1(\ux):\mc_2(\ux):-\mc_3(\ux)) \in \PP_3
\end{equation}  
lifts to a  regular map (so-called Bordiga conic bundle - see  \cite[Ex.~3 on p.~35]{alzatirusso})
$$
\Phi: \mbox{Bl}_{\mb}\PP_4 \rightarrow \PP_3.
$$
that 
is generically a conic-bundle ([ibid., Prop.~2.1]). The map 
$\Phi$ is the projection onto the second factor from the closure of the graph of the rational map defined by \eqref{eq-bordigarational}  (see also \eqref{eq-relacjaAB})
i.e. from the set
\begin{equation} \label{eq-rozdmuchaniebordygi-def}
\{ (\ux,y) \in \PP_4 \times \PP_3 :  \macB(y)\ux = 0 \}.
\end{equation}
By Lemma~\ref{lem-quadrics}.d it  has exactly ten $2$-dimensional
fibers over the points $y \in \PP_3$ such that $\rank(\macB(y)) = 2$. Such a fiber is the plane 
\begin{equation} \label{eq-fiber}
\Phi^{-1}(y) = \{ (\ux,y) :  \macB(y)\ux = 0 \}. 
\end{equation}
Observe that  restrictions of the cubics polynomials $\mc_{i}$ to the plane $\{ \macB(y)\ux = 0 \}$ are
proportional, so
the plane 
cuts $\mb$ along a cubic curve (see also \cite[Ex.~3 on p.~35]{alzatirusso}). \\
The remaining fibers 
$\Phi^{-1}(y)$ are 3-secant lines to $\mb$. They are given by \eqref{eq-fiber} with  $\rank(\macB(y)) = 3$.

In Sect.~\ref{sect-singularities}  we studied the map $\tx \lra X_5$.
By Lemma~\ref{nodesofquintic}
the quintic $X_{5}$ admits another small resolution of singularities
\begin{equation} \label{eq-rozdmuchaniebordygi}
\psi : \ttx \lra X_{5}
\end{equation}
obtained by blowing--up the Bordiga surface $\mb$. 
 The strict
transform $\tts$ of $\mb$ is a plane blown--up in $56$ 
points
(some of the $46$ points that are centers of the second blow-up 
may lie on the exceptional curves of the first blow--up). 
We  put $F_{1},\dots,F_{46}$ to  denote the exceptional  curves of the small resolution
in question. Then,  the two
resolutions differ by flops of the $46$ smooth rational curves  $L_{1},\dots,L_{46}\subset
\tx$ and $F_{1},\dots,F_{46}\subset \ttx$.

The restriction of the conic bundle $\Phi$ induces the regular map
\[\phi:\ttx \lra \PP_3.\]
This regular map is given by  the linear system $|3H_{1}-\tts|$ on
$\ttx$, where $H_{1}$ is pullback of the hyperplane section ${\mathcal O}_{\PP_4}(1)$.
We have the following intersection numbers

\begin{lemm}  \label{lem-intersh1is}
  \begin{eqnarray*}
    H_{1}^{3}=5,\\
    H_{1}^{2}\cdot\tts=6,\\
    H_{1}\cdot\tts^{2}=-2,\\
    \tts^{3}=-47,\\
    (3H_{1}-\tts)^{3}=2.
  \end{eqnarray*}
\end{lemm}
\begin{proof}
  The first two statements follow from the fact that $\deg(X_{5})=5$ and 
$\deg(\mb)=6$. The others can be obtained  from the equalities
\begin{equation} \label{eq-restrykcjedos1}
H_{1}|_{\tts}=4l-\sum_{1}^{10} \psi^*(\pi(E_{i})), \quad 
\tts|_{\tts}=-3l+\sum_{1}^{10}  \psi^*(\pi(E_{i})) + \sum_{1}^{46} F_{j}.
\end{equation}
where $l$ is  the pull-back of ${\mathcal O}_{\PP_{1}}(1)$ under both blow-ups.
Recall (Remark~\ref{rem-rankdwa}.b) that the curves $\pi(E_{1})$,$\ldots$, $\pi(E_{10})$  are lines on $\mb$.
\end{proof}
Since  $\phi$ is surjective, as an immediate consequence of Lemma~\ref{lem-intersh1is} we obtain 
\begin{cor} \label{cor-twotoone}
The mapping $\phi$ is generically 2:1.
\end{cor}

In order to obtain a precise description of fibers of $\phi$ we will
need the following lemma (compare \cite{michalek}): 

\begin{lemm} \label{lem-genericfiberofphi}
  A point $z \in \ttx$ is mapped by $\phi$ to $y\in\PP_{3}$ iff the
  3-space $\operatorname{span}((\psi(z):0:0:0),\Pl)$ is contained in the quadric
  $Q(y):=\sum_{i}y_{i}Q_{i}$. 
\end{lemm}
\begin{proof}  Observe that for any
  $x = (\ux:x_5:x_6:x_7)  \in \operatorname{span}((\ux:0:0:0),\Pl) $ we have 
\begin{equation} \label{eq-ciaglepowtarzana}
x^{T}\mq(y)x=\ux^{T}\underline{\mq}(y)\ux +2(x_{5},x_{6},x_{7})\macB(y)\ux
\end{equation}
\noindent
($\Leftarrow$): Put $\ux = \psi(z)$ in  \eqref{eq-ciaglepowtarzana} to  obtain 
$$
\psi(z)^{T}\underline{\mq}(y)\psi(z) =  -2(x_{5},x_{6},x_{7})\macB(y)\psi(z) \quad \mbox{ for all } x_5, x_6, x_7 \in \CC.
$$
The latter implies  $\macB(y)\psi(z)=0$ and (see \eqref{eq-fiber}) the equality $\phi(z) = y$.

\noindent
($\Rightarrow$):
Suppose that $z \in \ttx \setminus \tts$.
From $\phi(z)=y$ we get  $\macB(y)\psi(z)= 0$. By \eqref{eq-ciaglepowtarzana} we have
$$
x^{T}\mq(y)x=\psi(z)^{T}\underline{\mq}(y)\psi(z) \quad \mbox{ for all } x = (\psi(z):x_5:x_6:x_7)
\in \operatorname{span}(\psi(z),\Pl).
$$
But (see \eqref{eq-bordigarational}), we can assume that  $y = (\mc_0(\psi(z)):\ldots: -\mc_3(\psi(z)))$. Therefore,
Lemma~\ref{lem-equationofquintic}.a yields the equalities 
$\psi(z)^{T}\underline{\mq}(y)\psi(z) = \det(\AAA(\psi(z))) = 0$.
In this way we have shown the inclusion 
$$
\{ (\ux,y) \in \ttx  :  \macB(y)\ux = 0 \} \subset \{ (\ux,y) \in  \PP_4 \times \PP_3  :  \operatorname{span}((\ux:0:0:0),\Pl) \subset Q(y) \} \, , 
$$
which completes the proof.
\end{proof}

\newcommand{\hatl}{\hat{l}}
\newcommand{\hatE}{\hat{E}}
Recall, that we have the map $(\psi \circ (\pi|_{S})^{-1} \circ \sigma):  S_1  \rightarrow  \mb \backsimeq S \rightarrow  \Pl$.
In the lemma below we put $\hatl$ (resp. $\hatE_1$, $\ldots$, $\hatE_{10}$) 
to denote the pullback of  ${\mathcal O}_{\Pl}(1)$ (resp. of the exceptional divisors \eqref{eq-tenlines})
to $S_1$.

\newcommand{\planecubic}{C}   
\newcommand{\contractedcurve}{D}
\newcommand{\lineonbordiga}{l}

\begin{lemm} \label{lem-nocontractedcurvesonbordiga}
 An irreducible curve $\contractedcurve \subset \tts$ is contracted by $\phi$ iff (up to a relabelling of the divisors
$\hatE_1, \ldots, \hatE_{10}$ and $F_{1}, \ldots,F_{46}$)
it belongs to one of the following linear systems
\begin{itemize}
\item [a)]  $|\hatE_{1}-F_{1}-F_{2}-F_{3}-F_{4}|$,
\item [b)]  $|\hatl-\hatE_{1}-\hatE_{2}-\hatE_{3}-F_{1}-F_{2}-F_{3}|$,
\item [c)]  $|2\hatl-\hatE_{1}-\ldots-\hatE_{7}-F_{1}-F_{2}|$,
\item [d)]  $|3\hatl-2\hatE_{1}-\hatE_{2}-\ldots-\hatE_{9}-F_{1}-\ldots-F_{5}|$. 
\end{itemize} 
In the cases (a)--(c) the curve in question is the proper transform  of a line in $\mb$, whereas the case (d) corresponds
to a conic in the intersection of $\mb$ with the plane $\{\macB(y)\ux = 0\}$,  where $\rank(\macB(y))=2$. 
In particular, if the intersection $\mb \cap  \{\macB(y)\ux = 0\}$ is an irreducible cubic, then its proper transform 
is not contracted by $\phi$.
\end{lemm}
\begin{proof} 
Recall that  $\phi = \Phi|_{\ttx}$  and the fibers of $\Phi$ are lines and planes given by 
\eqref{eq-fiber}.

Before we prove the claim, we study two-dimensional fibers of $\Phi$.
Let $\sing(X_{16}) = \{P_1, \ldots, P_{10} \}$. By \eqref{eq-rankonlythree} 
for each singularity $P_i$ there exists a  unique point 
$y^{(i)} \in \PP_{3}$ such that $\macC(P_{i})y^{(i)}=0$. 
Then, by \eqref{eq-useful}, we have $\rank (\macB(y^{(i)}))=2$.  \\
Lemma~\ref{lem-quadrics}.a  yields that for each $i \in \{1, \ldots, 10\}$ there is a unique degree-three 
curve $C_{i} \subset \Pi$ such that  
$P_{j}\in C_{i}$, for $j\not=i$. Let 
$\tilde{C}_{i}:= \sigma^{*} C_{i}-\sum_{j\not=i}E_{j}\in|3l-\sum_{j\not=i}E_{j}|$ 
be the  corresponding
curve  on $S$. 
By direct computation the following equality holds
\begin{equation} \label{eq-planarcubic}
\pi(\tilde{C}_{i}) = \mb \cap \{\ux\in \PP_{4}:\macB(y^{(i)})\ux=0\}
\end{equation}

In general, cubics $C_{i}$ are smooth, and the 
curves $\pi(\tilde{C}_{i}) \subset \mb$ are also  smooth planar cubics. We have the following
possible degenerations: \\
(i) \, The curve $C_{i}$ is irreducible, but $\sing(C_{i}) = \{ P_{j_0} \}$ 
for a $j_0 \neq i$. Then the exceptional curve  $E_{j_0}$ is a
component of the curve $\tilde{C}_{i} :=\sigma^{*}C_{i}-\sum_{j\not=i}E_{j}$
and the curve  $\tilde{C}_{i} - E_{j_0}$ is irreducible. By Remark~\ref{rem-rankdwa}.b the image
$\pi(E_{j_0})$ is a line on $\mb$, whereas $\pi(\tilde{C}_{i})$ is a smooth conic. In this way we obtain a decomposition
of $\mb \cap \{\ux\in \PP_{4}:\macB(y^{(i)})\ux=0\}$. Observe that for a given integer $i \neq j_0$ there exists at most one cubic in 
$|{\mathcal O}_{\Pl}(3) - \sum_{j \neq i} E_j - E_{j_0}|$. \\
(ii) \, The cubic $\tilde C_{i}$ is  union of a  line and a smooth
conic. Then, by [\zanurzeniebordygi] and Lemma~\ref{lem-quadrics}.a the line contains two (resp. three)
singularities of $X_{16}$ and the  conic contains 7
(resp. 6) of them. \\
(iii) \, The curve $\tilde C_{i}$ can be union of three
lines. The assumption  [\zanurzeniebordygi]  yields that each line contains three singularities of
$X_{16}$. \\
In this way (up to a permutation of the points in $P_1, \ldots, P_9$), we obtain the following possibilities for the 
decomposition of the cubic \eqref{eq-planarcubic} for $i=10$:
\begin{eqnarray} \label{eq-possibledecompositions}
 &&  (3l-2E_{1}-E_{2}-\dots-E_{9})+E_{1}, \nonumber \\
 &&  (l-E_{1}-E_{2})+(2l-E_{3}-\dots- E_{9}), \nonumber \\
 &&  (l-E_{1}-E_{2}-E_{3})+(2l-E_{4}-\dots- E_{9}),  \\
 &&  (l-E_{1}-E_{2}-E_{3})+(2l-E_{3}-\dots- E_{9})+E_{3},  \nonumber \\
 &&  (l-E_{1}-E_{2}-E_{3})+(l-E_{4}-E_{5}-E_{6})+(l-E_{7}-E_{8}-E_{9}). \nonumber
\end{eqnarray}

After those preparations we can prove the lemma.
Assume that an irreducible curve $\contractedcurve \subset \tts$ is contained in $\phi^{-1}(y)$ for a point $y\in\PP_{3}$ .  
The map $\phi|_{S_1} \, : \, S_{1} \rightarrow \PP_{3}$ is given by the linear system 
\begin{equation} \label{eq-linearsystemforphi}
|15\hatl-4 \sum_1^{10} \hatE_{i}-\sum_1^{46} F_{j}|,
\end{equation}
so $D \neq F_j$ for each $j \leq 46$.

Suppose that  $\rank( \macB(y))=2$. We can assume that $\contractedcurve \subset \phi^{-1}(y^{(10)})$.
Then $\psi(\contractedcurve) \subset \mb$ is a component of \eqref{eq-planarcubic}.
If $\psi(\contractedcurve)$ is image under $\pi$ of a curve from the 
system 
 $|3l-2E_{1}-E_{3}-\dots-E_{9}|$, then we have 
$$
\deg(\psi(\contractedcurve)) = (3l-2E_{1}-E_{3}-\dots-E_{9})\cdot(4l-\sum_1^{10} E_{i}) = 12-2-8 =2.
$$
Let $\sing(X_{5}) \cap \psi(\contractedcurve) = \{\psi(F_1), \ldots, \psi(F_p)\}$.  Since $D$ coincides with 
the proper transform of $\psi(\contractedcurve)$ under the blow-up $\psi$, we have
$$
D \in |3\hatl-2\hatE_{1}-\hatE_{2}-\dots-\hatE_{9})-F_{1}-\dots-F_{p}|.
$$ 
and, by \eqref{eq-linearsystemforphi}, the degree of $\phi(D)$
is $(5-p)$. Consequently,  the curve $\contractedcurve$ is contracted by $\phi$ iff $p=5$. 

In the following table   we collect  data on each curve considered in \eqref{eq-possibledecompositions}. 
In particular,
the integer in the last column is the number of singularities of $X_{5}$ that lie on $\psi(D)$ provided  $D$ is
contracted  by the map $\phi$:
\begin{center}
\begin{tabular}{|r|c|c|}\hline
$|\pi^{-1}(\psi(\contractedcurve))|$  \hspace*{6ex}&  $\deg(\psi(\contractedcurve))$ & $\# (\sing(X_{5}) \cap \psi(\contractedcurve))$ \\\hline\hline
  $3l-2E_{1}-E_{2}-\dots-E_{9}$ & 2&5\\\hline
  $2l-E_{1}-\ldots-E_{6}$ & 2 & 6\\ \hline
  $2l-E_{1}-\ldots-E_{7}$ & 1 & 2\\ \hline
  $l-E_{1}-E_{2}$ & 2 & 7\\ \hline
  $l-E_{1}-E_{2}-E_{3}$ & 1 & 3\\ \hline
  $E_{1}$ & 1 & 4\\ \hline
\end{tabular}
\end{center}

Finally, observe that for  a point $y^{(i)}\in\PP_{3}$, where $i = 1, \ldots 10$, 
the intersection 
\begin{equation} \label{eq-wholeplanarquintic}
X_{5} \cap \{\ux\in \PP_{4}:\macB(y^{(i)})\ux=0\}
\end{equation}
is a degree--5 planar curve,
so it is union of the cubic considered above and a conic (possibly
reducible) that does not lie on $\mb$. 
The points $\psi(F_{j})$ are singular points of $X_{5}$, so they
are also singular points of the quintic curve \eqref{eq-wholeplanarquintic}, which
yields some extra constrains on the possible arrangements. Since a line contained in 
\eqref{eq-wholeplanarquintic}
intersects the residual quartic in four points, the line of the type
$(l-E_{1}-E_{2})$ is never contracted.  Similar argument rules out  the conic 
$(2l-E_{1}- \ldots -E_{6})$. In this way we arrive at the cases (a)--(d) of the lemma.

Assume that $\rank( \macB(y))=3$. Then $\contractedcurve$ is 
the strict transform of a line $\lineonbordiga_{y} \subset \mb$.
In particular, there exist $d, m_{i}, n_{j} \in \ZZ$ such that 
$D \in |d \hatl -\sum_1^{10} m_{i} \hatE_{i}-\sum_{1}^{46} n_{j}F_{j}|.$ 
Since the curve $D$ is smooth and rational, we have  $n_{j}=0$ or $1$. 
Moreover, by 
the genus formula 
$$(d \hatl - \sum_1^{10} m_{i} \hatE_{i} - \sum_1^{46} n_{j}F_{j})\cdot
((d-3) \hatl-\sum_1^{10} (m_{i}-1) \hatE_{i} -  \sum_1^{46} (n_{j}-1)F_{j}) =
d^{2}-3d-\sum_1^{10}(m_{i}^{2}-m_{i}) = -2.
$$
Furthermore, the equality $4d-\sum_1^{10} m_{i}=1$ holds
because $\lineonbordiga_{y}$ is a line on $\mb$
(see also Lemma~\ref{lem-equationofquintic}.b). 
Finally, since $D$ is contracted by the map given by the linear system $|3H_1 - S_1|$ we have 
$$
(15 \hatl-4\sum_1^{10} \hatE_{i}-\sum_{1}^{46} F_{j}) \cdot (d \hatl-\sum_1^{10} m_{i} \hatE_{i}-
\sum_1^{46} n_{j}F_{j}) =
15d-4\sum_1^{10} m_{i}-\sum_1^{46} n_{j}=0. 
$$
From the above we obtain the following equations
\begin{eqnarray*}
  &&\sum m_{i}^{2}=d^{2}+d+1,\\
  &&\sum m_{i}=4d-1,\\
  &&4-d=\sum n_{j},
\end{eqnarray*}
where $n_j = 0,1$. 
The solution $d = 3$, $m_1 =2$, $m_i = 1$ for $i >1$ is excluded by Lemma~\ref{lem-quadrics}.a.
The others correspond to
the cases (a)--(c) of the lemma.
\end{proof}

Now we are in position to prove
\begin{lemm} \label{lem-contractedcurves}
Let $y \in \PP_3$ be  a point such that $\rank(\macB(y)) = 3$. Then the fiber $\phi^{-1}(y)$
 is $1$-dimensional iff $\rank(\mq(y))=6$.
\end{lemm}
\begin{proof}
By abuse of notation we put $\psi$ to denote the blow-up $\mbox{Bl}_{\mb}\PP_4 \rightarrow \PP_4$.

Assume that the line $\Phi^{-1}(y)$ is contracted by $\phi$.
Then the set
$\psi(\Phi^{-1}(y))=\{\ux\in\PP_{4}:\macB(y)\ux=0\}$ is a line on $X_5$. 
Observe that the linear space $\mbox{span}(\{(\underline{x}:0:0:0) \, : \, \underline{x} \in  \psi(\Phi^{-1}(y)) \}, \Pl)$
is $4$-dimensional.
By Lemma~\ref{lem-genericfiberofphi} the quadric $Q(y)$ contains the
$4$--space
 $\mbox{span}(\{(\underline{x}:0:0:0) \, : \, \underline{x} \in  \psi(\Phi^{-1}(y)) \}, \Pl)$,
 which yields
$\rank(\mq(y))\le 6$. 
Finally  $\rank(\mq(y)) = 6$, because $\rank(\macB(y))=3$.

On the other hand, if $\rank(\mq(y))= 6$, then $\sing(Q(y))$ is a line.
Since  $\rank(\macB(y))= 3$,
the line   $\sing(Q(y))$ does not meet the plane $\Pl$. 
Put $L$ to denote
the image of the line  $\sing(Q(y))$ under  the projection from the plane $\Pl$. Then
$\operatorname{span}((\ux:0:0:0), \Pl) \subset Q(y)$ for every $\ux \in L$. From Lemma~\ref{lem-genericfiberofphi} we obtain that
the the proper transform of the line $L$ under the blow-up $\psi$ is contracted by $\phi$.  
\end{proof}
In the theorem below we identify curves in $\PP_4$ with their proper transforms  under the blow-up $\psi$:
whenever we say a line (resp. a conic) we mean its proper transform.
\begin{theo} \label{thm-fibers}
  There are four types of fibers $\phi^{-1}(y)$ of the map
  $\phi:\ttx\lra\PP_{3}$: 
  \begin{itemize}
  \item [a)] union of the conic residual to the cubic 
$\mb \cap \Phi^{-1}(y)$
in the planar quintic
$X_{5} \cap \Phi^{-1}(y)$ with the components of the cubic 
that satisfy the conditions of Lemma~\ref{lem-nocontractedcurvesonbordiga}
iff $\rank(\mq(y))\in \{5, 6, 7\}$  and   $\rank(\macB(y))=2$ 
(i.e. a singularity of  $Q(y)$ lies on $\Pl$),
 \item [b)] a line in $\PP_{4}$ iff $\rank(\mq(y))=6$ and   $\rank(\macB(y))=3$ 
(equivalently $\sing(Q(y)) \cap \Pl = \emptyset$),
 \item [c)] one point iff $\rank(\mq(y)) = 7$ and  $\rank(\macB(y))=3$,
  \item [d)] two points iff $\rank(\mq(y))=8$.
  \end{itemize}
\end{theo}
\begin{proof}
Suppose that $\rank(\macB(y))=3$. Then the linear space
 $\mbox{span}(\{(\underline{x}:0:0:0) \, : \, \underline{x} \in  \psi(\Phi^{-1}(y)) \}, \Pl)$
is $4$-dimensional and $\sing(Q(y)) \cap \Pl = \emptyset$. 
In view of Lemma~\ref{lem-contractedcurves},
we can 
assume that  $\rank(\mq(y)) \geq 7$ and
 the line $\psi(\Phi^{-1}(y)) =\{\ux:\macB(y)\ux=0\}$ is not contained in $X_{5}$. 

\noindent
Moreover, by   \eqref{eq-ciaglepowtarzana},  $\mbox{ for every point } x = (\ux, x_5, x_6, x_7) \, \in \, 
 \operatorname{span}(\{(\underline{x}:0:0:0) \, : \, \underline{x} \in  \psi(\Phi^{-1}(y)) \}, \Pl)$
  we have 
\begin{equation} \label{eq-malarestrykcja}
x^{T}\mq(y)x = \ux^{T}\underline{\mq}(y)\ux  \, .
\end{equation}
Observe, that the quadratic form given by  $\underline{\mq}(y)$ does not vanish identically on the line 
$\{\ux:\macB(y)\ux=0\}$ because the latter  is not contained in $X_{5}$.
Consequently, intersection of $Q(y)$ with the linear $4$-space 
$\operatorname{span}(\{(\underline{x}:0:0:0) \, : \, \underline{x} \, \in \,  \psi(\Phi^{-1}(y)) \}, \Pl)$
 consists of either one 
or two $3$-spaces.

Lemma~\ref{lem-genericfiberofphi} implies that 
the fibre  $\phi^{-1}(y)$  consists of  a unique point iff the restriction   
\begin{equation} \label{eq-megamegarestrykcja}
Q(y)|_{\operatorname{span}(\{(\underline{x}:0:0:0) \, : \, \underline{x} \, \in \,  \psi(\Phi^{-1}(y)) \}, \Pl)}
\end{equation}
is a full  square. \\
Suppose that the fibre in question is one point. From \eqref{eq-malarestrykcja} there exists a point
$\underline{v} \in \PP_5$, such that
$$
\macB(y)\underline{v} = 0 \quad \mbox{and} \quad   \underline{\mq}(y)\underline{v} = 0
$$
which means that $(\underline{v}:0:0:0) \in \sing(Q(y))$ and  $\rank(\mq(y))<8$. \\
Assume that  $\rank(\mq(y))<8$. Then 
 $Q(y)$ is a cone with the unique  vertex 
$(\underline{v}: v_5: v_6: v_7)$  away  from the plane $\Pl$. The latter yields $\underline{v} \neq 0$.
 Moreover, since the tangent space to $Q(y)$ in each point contains
the vertex we have 
$\macB(y)\underline{v} = 0$ and
 $$
(\underline{v}: v_5: v_6: v_7) \in  \operatorname{span}(\{(\underline{x}:0:0:0) \, : \, \underline{x} \in  \psi(\Phi^{-1}(y)) \}, \Pl)$$
Now $(\underline{v}: v_5: v_6: v_7)$ is a singularity 
of the restriction \eqref{eq-megamegarestrykcja}, so the polynomial
$\ux^{T}\underline{\mq}(y)\ux$ has a unique double root on the line $\{\ux:\macB(y)\ux=0\}$ and 
 \eqref{eq-megamegarestrykcja} is a full square.

Assume that $y\in\PP_{3}$ is a point such that $\rank(\macB(y))=2$,
and maintain the notation of the proof of Lemma~\ref{lem-nocontractedcurvesonbordiga}.
Then $y = y^{(i)}$ for an $i \in \{1, \ldots, 10\}$. 
By definition of the map $\phi$,  the proper transform under the blow-up $\psi$ of the (possibly reducible) conic residual to  \eqref{eq-planarcubic} in the quintic \eqref{eq-wholeplanarquintic}  is always contracted by $\phi$.
Moreover, a component of \eqref{eq-planarcubic} is contracted iff it satisfies the conditions
of Lemma~\ref{lem-nocontractedcurvesonbordiga}. \\
Observe that rank of the quadric $Q(y^{(i)})$ does not exceed $7$ because we have $\rank(\macB(y^{(i)}))=2$. 
\end{proof}

\begin{rem} \label{rem-atmostten}
By Lemma~\ref{lem-quadrics}.d there are exactly ten fibers of $\phi$ of the type a). The number of fibers of type b) will be discussed in the next section
(see Cor.~\ref{cor-singularities}).
\end{rem}

\section{Discriminant of the web $\WWW$} \label{discriminant}

In this section we maintain the notation and the assumptions of the previous ones. In particular we assume
that [\exactlytensingularities], [\zanurzeniebordygi], [\exactlyfortysixsingularities] hold.
Let $S_8$ stand for the discriminant surface of the web $\WWW$.  From now on we assume that

\vspace*{1ex}
\noindent
{\bf [\isolatedsingularities]:} {\sl the discriminant surface  $S_8$ has only isolated singularities }.

\vspace*{1ex}
To simplify notation we put
$$
{\mathbb I}_l :=  [a_{i,j}]_{i,j=0,\ldots,7},   \mbox{ where } a_{i,i} = 1 \mbox{ for  } i= 1, \ldots, l \mbox{ and } a_{i,j} = 0 \mbox{ otherwise.}
$$

At first we give conditions when  a singularity of $S_8$ is a node:
\begin{lemm} \label{lem-rankseven}
Let $Q_0$ be a rank-$7$ quadric in the web $\WWW$. 
\begin{itemize}
\item [a)] The quadric $Q_0$ is a smooth point of $S_8$ iff $\sing(Q_0) \notin \XX$.
\item [b)] The quadric $Q_0$ is a node of $S_8$ iff  $\sing(Q_0) \in \XX$.
\end{itemize}
\end{lemm}
\begin{proof}
Let $\mq_k =: [q_{i,j}^{(k)}]_{i,j=0,\ldots,7}$ and let ${\mathcal Q}^{(k)} := (q_{0,7}^{(k)}, \ldots, q_{6,7}^{(k)})$.
After an appropriate change of coordinates we can assume that $\mq_0 = {\mathbb I}_7$. In particular,
 $\sing(Q_0)  = \{(0:\ldots:0:1)\}$.

Let  ${\mathfrak G} := [{\mathfrak g}_{i,j}]_{i,j=1,2,3}$, where ${\mathfrak g}_{i,j} := \langle{\mathcal Q}^{(i)},{\mathcal Q}^{(j)}\rangle$
and $\langle\cdot,\cdot\rangle$ stands for the bilinear form defined by the identity matrix.
By direct computation we have
\begin{eqnarray*}
\det(\mq_0 + \sum_{k=1}^{3} \mu_k \cdot \mq_k ) = (\sum_{k = 1}^{3} \mu_k \cdot q_{7,7}^{(k)}) -  
     ((\mu_1,  \mu_2,  \mu_3) \, \, {\mathfrak G} \, \,  (\mu_1,  \mu_2,  \mu_3)^T) +
 (\mbox{terms of degree} \geq 3) \, .    
\end{eqnarray*}

\noindent
{\sl a)} Obviously, $(1:0:0:0)$ is a smooth point of $S_8$ iff the vector $(q_{7,7}^{(1)}, q_{7,7}^{(2)},  q_{7,7}^{(3)})$
does not vanish. The latter holds iff $(1:0:0:0) \notin \XX$, which concludes the proof.

\noindent
{\sl b)} ($\Rightarrow$): the implication in question results immediately from the part a).

\noindent
($\Leftarrow$):
Assume that $(q_{7,7}^{(1)}, q_{7,7}^{(2)},  q_{7,7}^{(3)}) = 0$. Then,   $Q_0 = (1:0:0:0) \in \sing(S_8)$
is a node iff the matrix ${\mathfrak G}$ has maximal rank, i.e.  ${\mathcal Q}^{(1)}$,  ${\mathcal Q}^{(2)}$,  ${\mathcal Q}^{(3)}$ 
are linearly independent. Moreover, we have  $(0:\ldots:0:1) \in \sing(\XX)$. 
 
Suppose that  $\rank(\mathfrak G) < 3$.
 Then, the last row in a matrix obtained as a  non-trivial  linear combination of the matrices $\mq_1, \mq_2, \mq_3$ vanishes,
 which means that the point $(0:\ldots:0:1)$ is a singularity of a quadric that belongs to
$\mbox{span}(\{Q_1, Q_2, Q_3\})$. In particular,  the quadric in question does not coincide with $Q_0$. The latter is impossible by 
Lemma~\ref{lem-quadrics}.b.
Contradiction.
\end{proof}
In the rank-$6$ case we have the following characterization.
\begin{lemm} \label{lem-ranksix}
Let $Q_0$ be a rank-$6$ quadric in the web $\WWW$. 
\begin{itemize}
\item [a)] The quadric $Q_0$ is a node of $S_8$ iff  $\sing(Q_0) \nsubseteq Q$ for all $Q \neq Q_0$, $Q \in \WWW$.
\item [b)]  $Q_0$ is an A$_m$ singularity, where $m \geq 2$, iff $\sing(Q_0) \cap \Pl = \emptyset$ and there exists a quadric $Q \in \WWW$, $Q \neq Q_0$
such that $\sing(Q_0) \subset Q$.
\item [c)] The quadric $Q_0$ is a double point of the surface $S_8$. 
\end{itemize}
\end{lemm}
\begin{proof} 
As in the proof of Lemma~\ref{lem-rankseven} we change the coordinates in such a way that $\mq_0 = {\mathbb I}_6$.
Then, the line $\sing(Q_0)$ is the set of zeroes of the coordinates $x_0, \ldots, x_5$.
Let $\langle\cdot,\cdot\rangle_{-}$ be the bilinear form on $\CC^3$ given by the formula:
\begin{equation} \label{eq-forma-h}
\langle(q_{6,,6}^{(1)}, q_{6,7}^{(1)},q_{7,7}^{(1)}),(q_{6,,6}^{(2)}, q_{6,7}^{(2)},q_{7,7}^{(2)})\rangle_{-} :=
1/2 \cdot ( q_{6,,6}^{(1)} \cdot q_{7,7}^{(2)} + q_{7,7}^{(1)} \cdot q_{6,,6}^{(2)} - 2 q_{6,7}^{(1)} q_{6,7}^{(2)}) \, 
\end{equation}
and let  ${\mathfrak H} := [{\mathfrak h}_{i,j}]_{i,j=1,2,3}$, where  
$
{\mathfrak h}_{i,j} := \langle(q_{6,,6}^{(i)}, q_{6,7}^{(i)},q_{7,7}^{(i)}),(q_{6,,6}^{(j)}, q_{6,7}^{(j)},q_{7,7}^{(j)})\rangle_{-} \, .
$
By direct computation we have
\begin{equation} \label{eq-taylorzH}
\det(\mq_0 + \sum_{k=1}^{3} \mu_k \cdot \mq_k ) = 
     ((\mu_1,  \mu_2,  \mu_3) \, \, {\mathfrak H} \, \,  (\mu_1,  \mu_2,  \mu_3)^T) +
 (\mbox{terms of degree} \geq 3) \, .    
\end{equation}

\noindent
{\sl a)} Observe that, by \eqref{eq-taylorzH}, the quadric $Q_0$ is a node of $S_8$ iff $\rank({\mathfrak H}) = 3$.

\noindent
($\Rightarrow$): Suppose that 
there exists a quadric $Q \neq Q_0$, $Q \in \WWW$ such that $\sing(Q_0) \subset Q$.
If $Q$ is given by the matrix  $[q_{i,j}]_{i,j=0,\ldots,7}$, then $q_{6,6}$, $q_{6,7}$, $q_{7,7}$ vanish, which yields that
$\rank({\mathfrak H}) < 3$.

\noindent
($\Leftarrow$): If $\rank({\mathfrak H}) < 3$, then we can find a matrix 
$\mq = [q_{i,j}]_{i,j=0,\ldots,7}$ such that $\mq \in \mbox{span}(\{\mq_1, \mq_2, \mq_3\})$ and 
the entries 
$q_{6,6}$, $q_{6,7}$, $q_{7,7}$ vanish.
The latter means that the quadric $Q$ given by $\mq$ contains the line $\sing(Q_0)$. 
We have $Q \neq Q_0$ because $Q_0 \notin \mbox{span}(\{Q_1, Q_2, Q_3\})$.

\noindent
{\sl b)} By part a)  we can assume that $\sing(Q_0) \subset Q_1$, which implies that 
 the entries 
$q_{6,6}^{(1)}$, $q_{6,7}^{(1)}$, $q_{7,7}^{(1)}$ of the matrix $\mq_1$ vanish.
Moreover, by \eqref{eq-taylorzH}, the quadric  $Q_0$ is an A$_m$ singularity, where $m \geq 2$, iff $\rank({\mathfrak H}) = 2$
(see e.g. \cite[Prop.~8.14]{malydimca}).

\noindent
($\Rightarrow$):
Suppose that $P \in \sing(Q_0) \cap \Pl$. Then $P \in \sing(\XX)$ and there exists
a quadric in the pencil $\mbox{span}(\{Q_2, Q_3\})$ that meets the line $\sing(Q_0)$ 
only in the point $P$. In particular we can assume that $Q_2 \cap \sing(Q_0) = \{ P \}$
and $P := (0:\ldots:0:1)$. The latter  yields 
$$
q_{6,,6}^{(2)}=1 \mbox{ and  } q_{6,,7}^{(2)} = q_{7,,7}^{(2)} = 0.
$$
Furthermore, since $P \in Q_3$ we have $q_{7,7}^{(3)} = 0$. Then  
$$
((\mu_1,  \mu_2,  \mu_3) \, \, {\mathfrak H} \, \,  (\mu_1,  \mu_2,  \mu_3)^T)=  - (q_{6,7}^{(3)})^2 \cdot \mu_3^2 \, , 
$$
which implies that $Q_0$ is not an A$_m$ singularity of the octic surface $S_8$.

\noindent
($\Leftarrow$):  
By Lemma~\ref{lem-ranksix}.a we have $\rank({\mathfrak H}) \leq 2$, so it suffices to show
that $\rank({\mathfrak H}) \notin \{0, 1\}$.  

Assume that $\rank({\mathfrak H}) = 1$. This means that 
\begin{equation} \label{eq-matrix-hadwadwa}
\operatorname{rank}
\left[
\begin{array}[c]{cc}
{\mathfrak h}_{2,2}   &   {\mathfrak h}_{2,3} \\
{\mathfrak h}_{3,2}   &   {\mathfrak h}_{3,3} \\ 
 \end{array}
\right]   = 1                \, .
\end{equation} 
Suppose that the vectors $(q_{6,,6}^{(2)}, q_{6,7}^{(2)},q_{7,7}^{(2)})$,
$(q_{6,,6}^{(3)}, q_{6,7}^{(3)},q_{7,7}^{(3)})$ are linearly independent.
By replacing $\mq_2$ with an appropriate linear combination of $\mq_2$, $\mq_3$ 
we can assume that the first column of the matrix \eqref{eq-matrix-hadwadwa} vanishes.
Then, from \eqref{eq-forma-h} and ${\mathfrak h}_{2,2} = 0$ 
we obtain the equality $\operatorname{rank}([q_{i,j}^{(2)}]_{i,j=6,7}) = 1$.
Performing an appropriate change of coordinates on the line $\sing(Q_0)$  
we arrive at 
\begin{equation}  \label{eq-normalizacjaQ2} 
q_{6,,6}^{(2)}=1 \mbox{ and } q_{6,,7}^{(2)} = q_{7,,7}^{(2)} = 0. 
\end{equation}
Then, the equality ${\mathfrak h}_{3,2} = 0$ 
yields $q_{7,7}^{(3)} = 0$. The latter implies that 
$$
(0:\ldots:0:1) \in \sing(Q_0) \cap  \sing(\XX) \, .
$$
Finally, the assumption    
[\exactlytensingularities] gives  $P \in \sing(Q_0) \cap \Pl$. \\
Suppose that \eqref{eq-matrix-hadwadwa} holds and the vectors $(q_{6,,6}^{(2)}, q_{6,7}^{(2)},q_{7,7}^{(2)})$,
$(q_{6,6}^{(3)}, q_{6,7}^{(3)},q_{7,7}^{(3)})$ are linearly dependent. Then, 
we can assume that the entries 
$q_{6,6}^{(2)}$, $q_{6,7}^{(2)}$, $q_{7,7}^{(2)}$ vanish, which implies  $\sing(Q_0) \subset Q_2$.
Finally, since
the line $\sing(Q_0)$ is contained in the quadrics $Q_1$, $Q_2$, each
point in the intersection $Q_3 \cap \sing(Q_0)$ is a singularity of $\XX$.
By [\exactlytensingularities] we have $\sing(Q_0) \cap \Pl \neq \emptyset$. 

In the same way  the equality $\rank({\mathfrak H}) = 0$ implies  $\sing(Q_0) \cap \Pl \neq \emptyset$. 
We omit the details.

\noindent
{\sl c)} By parts a) and b) we can assume that $\sing(Q_0) \subset Q_1$ and $\sing(Q_0) \cap \Pl \neq \emptyset$. 
Suppose that ${\mathfrak H} = 0$. From ${\mathfrak h}_{2,2} = 0$ we obtain  \eqref{eq-normalizacjaQ2}.
Then ${\mathfrak h}_{3,2} = 0$ yields $q_{7,7}^{(3)} = 0$, and by ${\mathfrak h}_{3,3} = 0$
the entry  $q_{6,6}^{(3)}$ vanishes. By replacing $\mq_3$ with $(\mq_3 - \mq_2)$ we obtain
the inclusion $\sing(Q_0) \subset Q_3$.

To complete the proof we assume, as in Section~\ref{sect-singularities} (see the proof of Remark~\ref{rem-noquadricsofrankfour}), that 
 the plane $\Pl$ (resp. the line $\sing(Q_0)$) is given by vanishing of the coordinates $x_0$,$\ldots$, $x_4$
(resp. $x_0$,$\ldots$, $x_3$ and $x_6, x_7$). Observe that  the point
$P = (0:\ldots:0:1:0:0) \in \sing(Q_0) \cap \Pl$ is a singularity of $\XX$. Therefore,  
Lemma~\ref{lem-quadrics}.b yields that 
the quadrics $Q_1$, $Q_2$, $Q_3$ are smooth in $P$. By direct computation, there
exist $v_1, \ldots, v_4 \in \CC$ such that 
the intersection of the tangent spaces  $\mbox{T}_{P}Q_1$,  $\mbox{T}_{P}Q_2$,  $\mbox{T}_{P}Q_3$ 
is parametrized by the map
$$
(\lambda_1, \lambda_2, \lambda_3, \lambda_4) \mapsto 
(\lambda_1 v_1, \lambda_1 v_2, \lambda_1 v_3, \lambda_1 v_4, \lambda_2, \lambda_3, \lambda_4) \, . 
$$ 
Substituting the above parametrization to (dehomogenized) $Q_0$ we see that the tangent cone
$\mbox{C}_{P}\XX$ is contained in union of two $3$-planes, so the point $P \in \XX$ is not a node.
Contradiction (see Lemma~\ref{lem-A1nodes}).
\end{proof}
\begin{rem} \label{remark-A3}
Direct computation with help of \cite{GPS01}, gives examples 
of webs of quadrics such that the assumptions 
[\exactlytensingularities],
[\zanurzeniebordygi],
[\exactlyfortysixsingularities]
[\isolatedsingularities] are fulfilled and the quadric $Q_0$ satisfies the conditions
of Lemma~\ref{lem-ranksix}.b. One can check that for generic choice of the quadrics  one obtains an A$_3$ singularity of the discriminant 
octic $S_8$.
\end{rem}
To complete the description of singularities of $S_8$ we prove the following lemma.
\begin{lemm}  \label{rem-rank-five}
A quadric $Q_0 \in \WWW$ is a point of multiplicity  at least $3$ on $S_8$ iff $\rank(\mq_0) = 5$.
\end{lemm} 
\begin{proof}
\noindent
($\Rightarrow$): Lemmata~\ref{lem-rankseven},~\ref{lem-ranksix}
imply  that $\rank(Q) \leq 5$. Remark~\ref{rem-noquadricsofrankfour} completes the proof. 

\noindent
($\Leftarrow$): Assume that $\mq_0 = {\mathbb I}_5$
and compute the determinant $\det(\mq_0 + \sum_{k=1}^{3} \mu_k \cdot \mq_k )$.
\end{proof}
The example below shows that the bound of  Remark~\ref{rem-noquadricsofrankfour} is sharp, and
the discriminant octic $S_8$ can have triple points. 
\begin{examp} \label{example-rankfive}
We define the following matrices:
\[
\mq_0 := \left[
\begin{array}{*{8}{r}}
\phantom{-}0 & \phantom{-}0 & 0 & 0 & 0 & \phantom{-}0 & 0 & 0 \\
0 & 0 & 0 & 0 & 0 & 0 & 0 & 0 \\
0 & 0 & 6 & 0 & -4 & 0 & -2 & 1 \\
0 & 0 & 0 & 4 & 3 & 0 & 2 & -4 \\
0 & 0 & -4 & 3 & 8 & 0 & -5 & 0 \\
0 & 0 & 0 & 0 & 0 & 0 & 0 & 0 \\
0 & 0 & -2 & 2 & -5 & 0 & 0 & 0 \\
0 & 0 & 1 & -4 & 0 & 0 & 0 & 0
\end{array}
\right]
\;  \mq_1 := \left[
\begin{array}{*{8}{r}}
-4 & -4 & 2 & -1 & 0 & -1 & -1 & -3 \\
-4 & 2 & 0 & 0 & 4 & -2 & 0 & -1 \\
2 & 0 & 0 & -1 & 2 & -2 & 4 & 2 \\
-1 & 0 & -1 & 2 & 3 & -1 & 3 & -2 \\
0 & 4 & 2 & 3 & -4 & -2 & 0 & 1 \\
-1 & -2 & -2 & -1 & -2 & 0 & 0 & 0 \\
-1 & 0 & 4 & 3 & 0 & 0 & 0 & 0 \\
-3 & -1 & 2 & -2 & 1 & 0 & 0 & 0
\end{array}
\right]
\]
\[
\mq_{2} := \left[
\begin{array}{*{8}{r}}
4 & -3 & -3 & -2 & 1 & -3 & -3 & -1 \\
-3 & -2 & -3 & -4 & 1 & 4 & 3 & 1 \\
-3 & -3 & 4 & 1 & 0 & 1 & 1 & 1 \\
-2 & -4 & 1 & 2 & -2 & 0 & 1 & 4 \\
1 & 1 & 0 & -2 & 4 & -1 & 0 & -1 \\
-3 & 4 & 1 & 0 & -1 & 0 & 0 & 0 \\
-3 & 3 & 1 & 1 & 0 & 0 & 0 & 0 \\
-1 & 1 & 1 & 4 & -1 & 0 & 0 & 0
\end{array}
\right] \;
     \mq_{3}:= \left[
\begin{array}{*{8}{r}}
4 & -1 & 2 & 2 & -2 & -1 & -2 & 0 \\
-1 & 2 & 2 & -3 & -1 & -4 & -2 & 4 \\
2 & 2 & -2 & -1 & 1 & 3 & 2 & -1 \\
2 & -3 & -1 & -2 & 0 & 1 & 3 & -2 \\
-2 & -1 & 1 & 0 & -4 & 4 & 1 & -1 \\
-1 & -4 & 3 & 1 & 4 & 0 & 0 & 0 \\
-2 & -2 & 2 & 3 & 1 & 0 & 0 & 0 \\
0 & 4 & -1 & -2 & -1 & 0 & 0 & 0
\end{array}
\right]
\]
By direct computation with help of \cite{GPS01}, the intersection  in $\PP_7$ of
the quadrics defined by the above matrices satisfies the assumptions 
[\exactlytensingularities], $\ldots$,  [\isolatedsingularities]. As one can easily see, we have
$\rank(\mq_0) = 5$.
\end{examp}

We put 
$\pi_2 \, : \, X_8 \rightarrow \WWW$ to denote the double cover of the web  $\WWW$  branched along the discriminant surface $S_8$. 
We have the following theorem (compare \cite[Thm~3.1]{michalek}).

\begin{theo} \label{thm-main}
Assume that [\exactlytensingularities], $\ldots$,  [\isolatedsingularities] hold.
 \begin{itemize}
\item [a)] There exists a (small) resolution $\hat{\phi}:\ttx \rightarrow X_8$ 
            of singularities of the double octic $X_8$ such that the following diagram commutes:
\begin{equation*} \label{eq-diagram}
\xymatrix{\ttx \ar[rr]^{\phi}\ar[dr]^{\hat{\phi}} & & \PP_3 &  \\
& X_8 \ar[ur]^{\pi_2} & }
\end{equation*}
\item [b)] Let $\pi$ be the map induced by the projection from the plane $\Pl$ (see \eqref{eq-podniesionaprojekcja})
and let $\sigma$ (resp. $\psi$) be the blow up defined by \eqref{eq-rozdmuchanieplaszczyzny} 
(resp. \eqref{eq-rozdmuchaniebordygi}). 
Then the composition
$$
\XX \,  \stackrel{\sigma^{-1}}{\plra}  \,  \tx  \, \stackrel{\pi}{\lra} \, X_5  \,  \stackrel{\psi^{-1}}{\plra}  \, 
\ttx  \stackrel{\hat{\phi}}{\lra} \, X_8
$$
is a birational map between the base locus of the web $\WWW$ and its double cover branched along the discriminant surface $S_8$. 
In particular, the base locus $\XX$ and the discriminant double octic $X_8$ are birational to the quintic $3$-fold 
$X_5$ (see \eqref{eq-quintic}) that contains Bordiga sextic.
\end{itemize}
\end{theo}
\begin{proof}
\noindent
a) Consider Stein factorization of the map $\phi: \ttx \rightarrow \PP_3$: 
$$
\phi = \hat{\phi} \circ \phi'
$$
where $\phi'$ is  finite and $\hat{\phi}$ has connected fibers. By Cor.~\ref{cor-twotoone} the map $\phi'$ is a 
(ramified) double cover of $\PP_3$. Thm~\ref{thm-fibers} and the assumption [\isolatedsingularities] imply the equality 
$\phi' = \pi_2$. Then the map $\hat{\phi}: \ttx \lra X_8$ is birational (see e.g. \cite[p.~11]{debarre}).
Thm~\ref{thm-fibers} implies that  
the set of $1$-dimensional fibers of the latter map coincides with $\hat{\phi}^{-1}(\sing(X_8))$. This completes the proof.

\noindent
b) We have just shown that the map $\hat{\phi}$ is birational.  
The claim follows from Lemma~\ref{lem-X5isnormal}.a.
\end{proof}

In the case of the double sextic defined by a net of quadrics that contain a (fixed) line
the discriminant curve has only nodes as singularities (see \cite[Thm~3.3]{cynkrams}).

In the corollary below we discuss the singularities
of the discriminant surface $S_8$.

\begin{cor} \label{cor-singularities}
Assume that [\exactlytensingularities], $\ldots$,  [\isolatedsingularities] hold.
\begin{itemize}
\item [a)] 
The  equality $\sum_{P \in \sing(X_8)} (\mu(P, X_8) + 1)  = 188 $    holds,
where $\mu(P, X_8)$ stands for the Milnor number of $X_8$ in the point $P$.
 \item [b)] A quadric  $Q_0 \in \WWW$ is a singularity of $S_8$ of the type given in the first column
of the table below  iff it satisfies the conditions listed in the other column

\vspace*{3ex}
\hspace*{-10ex}
\begin{tabular}{|c|c|c|c|} 
\hline
Type of singularity    & \multicolumn{3}{c|}{Conditions} \\ \cline{2-4}
                              & $\operatorname{rank}(\mq_0)$ &                                      &                             \\
                                                                                                        
\hline 
smooth point                  &           $7$              &   $\sing(Q_0) \cap \XX = \emptyset$   &                             \\  
\hline  
A$_1$                         &           $7$              & $\sing(Q_0) \cap \XX \neq \emptyset$  &                             \\
\cline{2-4}
                              &           $6$              &                                       & 
 $\{Q \in \WWW \, : \, Q \neq Q_0, \operatorname{sing}(Q_0) \subset Q \} = \emptyset$    \\
\hline
A$_m$, $m \geq 3$, $m$ odd    &           $6$              &   $\sing(Q_0) \cap \Pl = \emptyset$   &   $\{Q \in \WWW \, : \, Q \neq Q_0, \operatorname{sing}(Q_0) \subset Q \} \neq \emptyset$          \\     
\hline
double point  of corank $2$   &           $6$     &        $\sing(Q_0) \cap \Pl \neq \emptyset$  &   
$\{Q \in \WWW \, : \, Q \neq Q_0, \operatorname{sing}(Q_0) \subset Q \} \neq \emptyset$          \\
\hline
$k$-fold point,   $k \geq 3$     &           $5$     &    &                                                                             \\ 
\hline
\end{tabular}
\end{itemize}
\end{cor}
\begin{proof}

\noindent
a)  To compute the sum of Milnor numbers of
singularities of $X_8$ we compare topological Euler numbers of $\ttx$  and $X_8$. By the assumption 
[\exactlyfortysixsingularities] and Lemma~\ref{nodesofquintic} we have $e(\ttx) = - 108$.
On the other hand, by Chern class argument the  Euler number of a smooth octic in $\PP_3$ is $304$, so  
\cite[Cor.~5.4.4]{dimca} implies $e(X_8) = - 296 + \sum_{P \in \sing(X_8)} \mu(P, S_8)$. 
Observe that in our set-up the equality $\mu(P, S_8) = \mu(P, X_8)$ holds.
From Thm~\ref{thm-main}.a we get 
\begin{equation}
- 108 + \#(\sing(X_8) = - 296 + \sum_{P \in \sing(X_8)} \mu(P, X_8) \, .
\end{equation}
that yields the claim.

\noindent
b) By Thm~\ref{thm-main}.b and \cite[Cor.~1.16]{reid-pagoda}  the
octic  $S_8$ has no A$_m$ points with $m$ even. 
The claim follows now directly from Lemmata~\ref{lem-rankseven},~\ref{lem-ranksix}, and Lemma~\ref{rem-rank-five}. 
\end{proof}
\begin{rem} 
Under the assumptions  [\exactlytensingularities], $\ldots$,  [\isolatedsingularities] the following 
inequality holds
$$
\# \{ P \in \sing(S_8) \, : \, P \mbox{ is not an A}_m \text{ point, where } m \geq 1 \}   \leq 10.
$$ 
\end{rem}
\begin{proof} 
By Lemmata~\ref{lem-rankseven},~\ref{lem-ranksix}
each double point $Q_0 \in \sing(S_8)$ that is not an A$_m$ singularity is a singular quadric and 
its singular locus meets the plane $\Pl$.
The same holds for rank-$5$ quadrics in the web $\WWW$ (see Thm~\ref{thm-fibers}). 
Therefore, the inequality results from Remark~\ref{rem-atmostten}.
\end{proof}

\noindent
{\sl Final remarks:} a) According to  \cite[Thm~4.1]{laufer} the normal bundle a 
smooth rational curve that is contracted on a $3$-fold  is one of the following:
 $({\mathcal O}_{\PP_1}(-1) \oplus {\mathcal O}_{\PP_1}(-1))$,  $({\mathcal O}_{\PP_1}(-2) \oplus {\mathcal O}_{\PP_1})$, 
$({\mathcal O}_{\PP_1}(-3) \oplus {\mathcal O}_{\PP_1}(1))$. Remark~\ref{remark-A3} and Ex.~\ref{example-rankfive}
show that all such bundles can come up in our set-up. For the conditions imposed on the equation
of a (smooth) $3$-fold quintic in $\PP_4$ by the normal bundle of a contracted curve 
the reader should consult \cite[App.~A, B]{katz}.   

\noindent
b) Assume that all singularities of $S_8$ are A-D-E points. By \cite[Thm~1.1]{batyrev}
the Hodge diamond  of any  small K\"ahler resolution of the double octic $X_8$ 
coincides with the one given   in Lemma~\ref{lem-hn}. In view of \cite[Cor.~5.1]{rams-hab}
and [ibid.,~Prop.~6.1], the latter implies that the assumptions 
 [\exactlytensingularities], $\ldots$,  [\isolatedsingularities]  
determine position of singularities of $S_8$ with respect to 
sections of ${\mathcal O}_{\PP_3}(8)$  
(compare \cite[Prop.~2.13]{michalek}). 

\noindent
c) In Thm~\ref{thm-fibers} we describe components of  $\Phi^{-1}(y)$ when
$\rank(\macB(y)) = 2$.
Since all singularities of $X_8$ admit a small resolution, \cite[Thm~5.5]{Morrison}
can be applied to obtain a more precise description of such fibers. 
We omit details because of lack of space.

\vspace*{2ex}
\noindent
{\it Acknowledgement.} Part of this paper was written during the second author's stay at the Department of Mathematics of N\"urnberg-Erlangen University. 
The second author would like to thank Prof.~W.~P.~Barth
for his help and encouragement.

\parindent=0cm
\end{document}